\documentclass[11pt,a4paper]{article}
\usepackage{mycommands}

\setgcgap{0.08}  % global default for this paper

\newcommand{\opmo}{$0$/$\pm1$}
\newcommand{\drm}{M^{\!\times2}}
\newcommand{\grm}{G^{\times\!}}
\newcommand{\tmm}{\mathbb{M}(M)}

\title{\textbf{Growth rates of \\
geometric grid classes of permutations
}}

\author{$\phantom{{}^\dagger}$David Bevan${}^\dagger$}

\date{}

\begin{document}
\maketitle

{\let\thefootnote\relax\footnotetext
{${}^\dagger$Department of Mathematics and Statistics, The Open University, Milton Keynes, England.}}

{\let\thefootnote\relax}

\begin{abstract}
\noindent
Geometric grid classes of permutations have proven to be %a
key in investigations of
%properties of
%enumerative and structural properties of
classical permutation pattern classes.
%, especially when considering properties related to their growth rates.
By considering the representation of
gridded permutations as words in a trace monoid,
we prove that
every
geometric grid class
has a
growth rate %,
%and that it
which is given by %equals
the square of the largest root of the matching polynomial of
a related graph.
%(the row-column graph of the double refinement of its matrix),
%and is dependent on the parity of its cycles.
%\todo{extend}
As a consequence, we characterise the set of growth rates of geometric grid classes in terms of the spectral radii of trees,
explore the influence of ``cycle parity'' on the growth rate, compare the growth rates of geometric grid classes against those of the corresponding
monotone grid classes, and present new results concerning the effect of edge subdivision on the largest root of the matching polynomial.
\end{abstract}

% ================================================================
\section{Introduction}

Following the proof by Marcus \& Tardos~\cite{MT2004} of the Stanley--Wilf conjecture, there has been particular interest in the growth rates of permutation classes. % and in the properties of classes with small growth rates.
%In particular,
Kaiser \& Klazar~\cite{KK2003} determined the possible growth rates less than $2$,
and then
Vatter~\cite{Vatter2011} characterised all the (countably many) permutation classes with growth rates below $\kappa\approx2.20557$ and
established that there are uncountably many permutation classes with growth rate $\kappa$.
%Building on additional work by Vatter~\cite{Vatter2010b}, the current author~\cite{Bevan2014b} determined that there are permutation classes having every growth rate above $\lambda_B\approx2.35698$.
%(The behaviour between $\kappa$ and $\lambda$ is the subject of current research.) %\footnote{It is now known that there are permutation classes having every growth rate above $\lambda_2\approx2.36938$.}) %remains an open question.)
%In addition, Albert, Ru\v{s}kuc \& Vatter~\cite{ARV2012} have proved that every permutation class with growth rate less than $\kappa$ has a rational generating function.
Critical to these results has been the consideration of \emph{grid classes} of permutations, and particularly of \emph{geometric} grid classes.
% of permutations.
Geometric grid classes have also been used to achieve the enumeration of some specific permutation classes~\cite{AAB2012,AAV2014}.
Following initial work
on particular geometric grid classes
by Waton~\cite{WatonThesis},
Vatter \& Waton~\cite{VW2011b},
and Elizalde~\cite{Elizalde2011},
their general structural properties have been
investigated
in articles
by Albert, Atkinson, Bouvel, Ru\v{s}kuc \& Vatter~\cite{AABRV2011}.
and Albert, Ru\v{s}kuc \& Vatter~\cite{ARV2012}.
We build on their work to establish the growth rate of any given geometric grid class.
Before we can state our result, we need a number of definitions.
%(these can all be found in \cite{AABRV2011});
%(see \cite{AABRV2011} for a more detailed presentation). %;
%we begin with geometric grid classes.

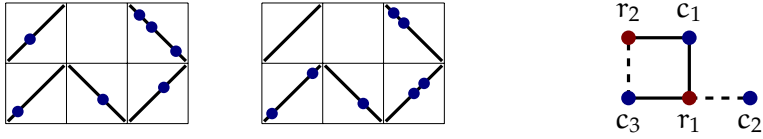
\begin{figure}[ht]
$$
%{\setgcscale{0.4} \gctwo{3}{1,0,-1}{1,-1,1}}
%\qquad\qquad
{
\setgcscale{0.8}
\setgcgap{0.037}
\setgcptgridscale{5}
\setgcptsize{0.095}%{0.105}
\setgcptcolor{blue!50!black}
\gctwo[-1, 1,7,-1,-1, -1, -1,-1,2,-1, -1, 9,8,3,6]{3}{1,0,-1}{1,-1,1}
\qquad
\setgcptgridscale{6}
\gctwo[-1, 1,-1,-1,-1,5, -1, -1,-1,-1,2,-1, -1, 11,10,3,4]{3}{1,0,-1}{1,-1,1}
}
\qquad\qquad\quad
 \raisebox{-.15in}{
  \begin{tikzpicture}[scale=0.8]
      \draw [very thick] (0,0)--(1,0)--(1,1)--(0,1);
      \draw [very thick,dashed] (0,0)--(0,1);
      \draw [very thick,dashed] (1,0)--(2,0);
      \fill[radius=0.12,black!50!blue] (0,0) circle ;
      \fill[radius=0.12,black!50!red] (1,0) circle ;
      \fill[radius=0.12,black!50!blue] (2,0) circle;
      \fill[radius=0.12,black!50!red] (0,1) circle ;
      \fill[radius=0.12,black!50!blue] (1,1) circle ;
      \node[below]at(0,-.1){$c_3$};
      \node[below]at(1,-.1){$r_1$};
      \node[below]at(2,-.1){$c_2$};
      \node[above]at(0,1.1){$r_2$};
      \node[above]at(1,1.1){$c_1$};
    \end{tikzpicture}}
$$
\caption[figGeomClass]{
At left: The standard figure for
$\begin{geommx}1&\pos0&-1\\ 1&-1&\pos1 \end{geommx}$,
showing two %distinct
%griddings
plots
of the permutation $1527634$
with distinct griddings.
%in this class,
At right: Its row-column graph;
% with a single negative cycle;
positive edges are shown as solid lines, negative edges are dashed.
}\label{figGeomClass}
\end{figure}
A geometric grid class is specified by a \opmo{} matrix which represents the shape of plots of permutations in the class.
To match the Cartesian coordinate system, we index these matrices from the lower left, by column and then by row.
If $M$ is such a matrix, then we say that the \emph{standard figure} of $M$, denoted  $\Lambda_M$, is the subset of $\mathbb{R}^2$ consisting of the
%disjoint
union of
oblique
open line segments $L_{i,j}$ with slope $M_{i,j}$ for each $i,j$ for which
%$M_{i,j}\neq0$,
$M_{i,j}$ is nonzero,
where $L_{i,j}$ extends
%where the endpoints of $L_{i,j}$ are
from $(i-1,j-1)$ to $(i,j)$
%$(i-1,j-1)$ and $(i,j)$
if $M_{i,j}=1$, and
from $(i-1,j)$ to $(i,j-1)$
%$(i-1,j)$ and $(i,j-1)$
if $M_{i,j}=-1$.
%the open line segment from $(i-1,j-1)$ to $(i,j)$ for each $i,j$ for which $M_{i,j}=1$ and
%the open line segment from $(i-1,j)$ to $(i,j-1)$ for each $i,j$ for which $M_{i,j}=-1$.
The \emph{geometric grid class} $\Geom(M)$ is
then defined to be
the set of
%the %those
permutations $\sigma_1\sigma_2\ldots\sigma_n$ that can be plotted as a subset of the standard figure, i.e.~for which there exists a sequence of points $(x_1,y_1),\ldots,(x_n,y_n)\in\Lambda_M$ %,
%known as a \textbf{plot},
such that $x_1<x_2<\ldots<x_n$ and the sequence $y_1,\ldots,y_n$ is order-isomorphic to $\sigma_1,\ldots,\sigma_n$.
See Figure~\ref{figGeomClass} for an example.

If $g_n$ is the number of permutations of length $n$ in
$\Geom(M)$,
then the
\emph{growth rate}
of the class is given by
$\gr(\Geom(M))=\liminfty g_n^{\,1/n}$.
We will demonstrate that this limit exists\footnote{It is widely believed that all permutation classes have growth rates.
The proof of the Stanley--Wilf conjecture by
Marcus \& Tardos~\cite{MT2004} establishes only that each has an \emph{upper} growth rate ($\limsup g_n^{\,1/n}$).} and determine its value for any given \opmo{} matrix $M$.

Much of the structure of a geometric grid class is reflected in a %bipartite
graph that
we associate with the underlying matrix.
If $M$ is a \opmo{} matrix of dimensions $t\ttimes u$, the \emph{row-column graph} $G(M)$ of $M$ is the bipartite graph
with vertices $r_1,\ldots,r_t,c_1,\ldots,c_u$ and an edge between $r_i$ and $c_j$ if and only if $M_{i,j}\neq0$.
We label each edge $r_ic_j$ with the value of $M_{i,j}$.
Edges labelled $+1$ are called \emph{positive}; edges labelled $-1$ are called \emph{negative}.
See Figure~\ref{figGeomClass} for an example.

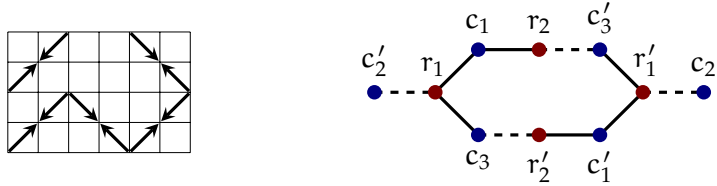
\begin{figure}[ht]
$$
{
\setgcscale{0.4}
\setgcgap{0.04}
\setgcarrowmode
\setgcarrowtip{stealth}
\gcfour{6}{0,3,0,0,-2}{2,0,0,0,0,-3}{0,3,-2,0,0,3}{2,0,0,-3,2}
}
\qquad\quad\qquad
 \raisebox{-.3625in}{
  \begin{tikzpicture}[scale=0.8]
      \draw [very thick] (1.707,-0.707)--(1,0)--(1.707,0.707)--(2.707,0.707);
      \draw [very thick] (2.707,-0.707)--(3.707,-0.707)--(4.414,0)--(3.707,0.707);
      \draw [very thick,dashed] (0,0)--(1,0);
      \draw [very thick,dashed] (1.707,-0.707)--(2.707,-0.707);
      \draw [very thick,dashed] (2.707,0.707)--(3.707,0.707);
      \draw [very thick,dashed] (4.414,0)--(5.414,0);
      \fill[radius=0.12,black!50!blue] (0,0) circle ;
      \fill[radius=0.12,black!50!red] (1,0) circle ;
      \fill[radius=0.12,black!50!blue] (1.707,-0.707) circle;
      \fill[radius=0.12,black!50!blue] (1.707,0.707) circle;
      \fill[radius=0.12,black!50!red] (2.707,-0.707) circle ;
      \fill[radius=0.12,black!50!red] (2.707,0.707) circle ;
      \fill[radius=0.12,black!50!blue] (3.707,-0.707) circle;
      \fill[radius=0.12,black!50!blue] (3.707,0.707) circle;
      \fill[radius=0.12,black!50!red] (4.414,0) circle ;
      \fill[radius=0.12,black!50!blue] (5.414,0) circle ;
      \node[above]at(0,.1){$c'_2$};
      \node[above]at(1,.1){$r_1\,$};
      \node[above]at(1.707,0.807){$c_1$};
      \node[above]at(2.707,0.807){$r_2$};
      \node[above]at(3.707,0.807){$c'_3$};
      \node[above]at(4.414,.1){$\,r'_1$};
      \node[above]at(5.414,.1){$c_2$};
      \node[below]at(1.707,-.807){$c_3$};
      \node[below]at(2.707,-.807){$r'_2$};
      \node[below]at(3.707,-.807){$c'_1$};
    \end{tikzpicture}}
$$
\caption[figDoubleRefinement]{At left: The standard figure of
%the double refinement of $\begin{smallmx}1&\pos0&-1\\ 1&-1&\pos1 \end{smallmx}$,
$\begin{smallmx}1&\pos0&-1\\ 1&-1&\pos1 \end{smallmx}^{\!\times2}$,
with a consistent orientation marked. At right:
Its row-column graph.
}\label{figDoubleRefinement}
\end{figure}
We need one final definition related to geometric grid classes.
If $M$ is a \opmo{} matrix of dimensions $t\ttimes u$, we define the \emph{double refinement} $\drm$ of $M$ to be the \opmo{} matrix of dimensions $2\+t\ttimes 2\+u$ obtained from $M$ by replacing each $0$ with $\begin{smallmx}0&0\\0&0\end{smallmx}$, each $1$ with $\begin{smallmx}0&1\\1&0\end{smallmx}$, and each $-1$ with $\begin{smallmx}-1&\pos0\\ \pos0&-1\end{smallmx}$.
See Figure~\ref{figDoubleRefinement} for an example.
Note that the standard figure of $\drm$ is essentially a scaled copy of the standard figure of $M$, so we have:
\begin{obs}\label{obsGeomDouble}
$\Geom(\drm)=\Geom(M)$ for any \opmo{} matrix $M$.
\end{obs}
%\vspace{-6pt}
We will demonstrate a connection between the growth rate of
%geometric grid class
$\Geom(M)$ and the \emph{matching polynomial} of the graph $G(\drm)$, the row-column graph of the double refinement of $M$.
A~$k$-\emph{matching} of a graph is a set of $k$ edges, no
pair of which have a vertex in common.
For example, the negative (dashed) edges in the graph in Figure~\ref{figDoubleRefinement} constitute a $4$-matching.
%A matching is also known as an \emph{independent edge set}.
If, for each~$k$, $m_k(G)$ denotes the number of distinct $k$-matchings of a graph $G$
%, and $G$ has $n$ vertices,
with $n$ vertices,
then the \emph{matching polynomial}
$\mu_G(z)$
%$\mu(G;z)$
of $G$ is defined to be
\begin{equation}\label{eqMatchingPolynomialDefn}
%$$
\mu_G(z)
%\mu(G;z)
\;=\; \sum_{k=0}^{\floor{n/2}} (-1)^k\+ m_k(G)\+z^{n-2k}.
%$$
\end{equation}
Observe that the exponents of the variable $z$ enumerate \emph{defects} in $k$-matchings: the number of vertices which are \emph{not} endvertices of an edge in such a matching.
If $n$ is even, $\mu_G(z)$ is an even function;
if $n$ is odd, $\mu_G(z)$ is an odd function.

With the relevant definitions complete, we can now state our theorem:

\thmbox{
\begin{thm}\label{thmGeomClassGrowthRate}
The growth rate of
geometric grid class
$\Geom(M)$
exists and
is equal to
%equals
the square
of the largest root of
the matching polynomial
$\mu_{G(\drm)}(z)$,
where
$G(\drm)$ is
the row-column graph of the double refinement of $M$.
\end{thm}
} %\thmbox

In the next section, we prove this theorem by utilizing the link between geometric grid classes and trace monoids, and their connection to rook numbers and the matching polynomial.
Then, in Section~\ref{SectGeomImplications}
we investigate a number of implications of this result by utilizing properties of the matching polynomial, especially the fact that the moments of $\mu_G(z)$
enumerate certain closed walks on $G$.
Firstly, we characterise the growth rates of geometric grid classes in terms of the spectral radii of trees.
Then, we explore the influence of cycle parity on growth rates and relate
the growth rates of geometric grid classes to those of monotone grid classes.
Finally, we consider the effect
of subdividing edges in the row-column graph, proving some new results regarding how edge subdivision affects the largest root of the matching polynomial.

%\newpage
% ================================================================
%\section{Griddings, consistent orientations and trace monoids}
\section{Proof of Theorem~\ref{thmGeomClassGrowthRate}}

In order to prove our result, we make use of the
%begin by relating
connection between geometric grid classes and \emph{trace monoids}. This relationship was first used by Vatter \& Waton~\cite{VW2011} to establish certain structural properties of
%some
grid classes, and was developed further in~\cite{AABRV2011} from where we use a number of results.
To begin with, we need to consider \emph{griddings} of permutations.

If $M$ has dimensions $t\ttimes u$, then an $M$-\emph{gridding} of a permutation $\sigma_1\ldots\sigma_n$ in $\Geom(M)$ %of length $n$
consists of two sequences $c_1,\ldots,c_t$ and $r_1,\ldots,r_u$
%for which
such that
there is some plot $(x_1,y_1),\ldots,(x_n,y_n)$ of $\sigma$
%such that
for which
$c_i$ is the
number of points $(x_k,y_k)$ in column $i$
(with $i-1<x_k<i$),
and $r_j$ is the
number of points in row $j$
(with $j-1<y_k<j$).\footnote{This
%is equivalent to the traditional
definition
of an $M$-gridding
is equivalent to the traditional one
%which is usually
given
in terms
%as sequences
of the positions of the cell dividers
relative to the points
%in a plot of the points
$(k,\sigma_k)$.}
Note that a permutation may have multiple distinct griddings in a given geometric grid class; see Figure~\ref{figGeomClass} for an example.
We call a permutation together with
%a specific $M$-gridding
one of its $M$-griddings
an
$M$-\emph{gridded permutation}.
%Given a \opmo{} matrix, we
We use $\Geomhash(M)$ to denote the set of all $M$-gridded permutations.

From an enumerative perspective,
it can be much easier working with $M$-gridded permutations than directly with the permutations themselves.
The following observation means that we can, in fact, restrict our considerations to $M$-gridded permutations:
\begin{lemma}[see Vatter~\cite{Vatter2011} Proposition~2.1]\label{lemmaGRGriddings}
If it exists, the growth rate of
$\Geom(M)$ is equal to the growth rate of the corresponding class of $M$-gridded permutations $\Geomhash(M)$.
\end{lemma}
\begin{proof}
Suppose that $M$ has dimensions $t\ttimes u$.
Each permutation in $\Geom(M)$ has at least one gridding in $\Geomhash(M)$, but no permutation of length $n$ in $\Geom(M)$ can have more than $\binom{n+t-1}{t-1}\binom{n+u-1}{u-1}$ griddings in $\Geomhash(M)$ because
that is the number of ways of choosing the number of points in each column and row.
Thus the number of $M$-gridded permutations of length $n$ is no more than a polynomial multiple of the number of $n$-permutations in $\Geom(M)$; the result follows immediately from the definition of the growth rate.
\end{proof}
To determine the growth rate of
$\Geomhash(M)$,
we will relate $M$-gridded permutations to words in a trace monoid.
To achieve this, one additional concept is required, that of a \emph{consistent orientation} of a standard figure.
If $\Lambda_M=\bigcup\,\{L_{i,j}:M_{i,j}\neq0\}$ is the standard figure of a \opmo{} matrix $M$, then a \emph{consistent orientation} of $\Lambda_M$ consists of an orientation of each oblique line $L_{i,j}$ such that in each column either all the lines are oriented leftwards or all are oriented rightwards, and in each row either all the lines are oriented downwards or all are oriented upwards.\footnote{For
ease
%simplicity
of exposition, we
use
%prefer
the concept of a consistent orientation
rather than the
%to the
%equivalent traditional
approach used previously involving
partial multiplication matrices; results from \cite{AABRV2011} follow \emph{mutatis mutandis}.}
See Figures~\ref{figDoubleRefinement} and \ref{figPermWord} for examples.

It is not always possible to consistently orient a standard figure.
The ability to do so depends on the \emph{cycles} in the row-column graph.
We say that the \emph{parity} of a cycle in $G(M)$ is
the product of the labels of its edges,
a \emph{positive cycle} is one which has parity $+1$, and
a \emph{negative cycle} is one with parity $-1$.
The following result relates
cycle parity to
consistent orientations:
\begin{lemma}[see Vatter \& Waton~\cite{VW2011} Proposition~2.1]\label{lemmaNegativeCycles}
The standard figure
%of a \opmo{} matrix $M$
$\Lambda_M$
has a consistent orientation if and
only if its row-column graph $G(M)$ contains no negative cycles.
\end{lemma}
For example,
$\begin{rowcolmx}1&\pos0&-1\\ 1&-1&\pos1 \end{rowcolmx}$
contains a negative cycle so its standard figure
has no
consistent orientation (see Figure~\ref{figGeomClass}), whereas
$\begin{rowcolmx}-1&\pos0&-1\\ \pos1&-1&\pos1 \end{rowcolmx}$
has no negative cycles so its standard figure has a consistent orientation
(see Figure~\ref{figPermWord}).

On the other hand, we can always consistently orient the standard figure of the double refinement of a matrix
by orienting each oblique line towards the centre of its $2\times2$ block (as in Figure~\ref{figDoubleRefinement}).
So
we have the following:
\begin{lemma}[see \cite{AABRV2011} Proposition~4.1]\label{lemmaDoubleRefinementConsistentOrientation}
If $M$ is any \opmo{} matrix, then $\Lambda_{\drm}$ has a consistent orientation.
\end{lemma}
Thus, by Lemma~\ref{lemmaNegativeCycles}, the row-column graph of the double refinement of a matrix never contains a negative cycle.
Figure~\ref{figDoubleRefinement} shows a consistent orientation of the standard figure of the double refinement of a matrix whose standard figure (shown in Figure~\ref{figGeomClass}) doesn't itself have a consistent orientation.

\begin{figure}[ht]
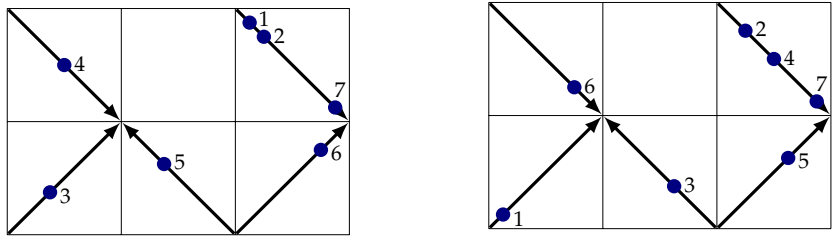

$$
{
\setgcscale{1.5}
\setgcgap{0.01}
\setgcptgridscale{8}
\setgcptsize{0.063}
\setgcptcolor{blue!50!black}
\setgcarrowmode
%\setgcarrowtip{angle 90}
\setgcextra{
\node[right]at(2.115,1.865){${}^1$};
\node[right]at(2.25,1.71){${}^2$};
\node[right]at(0.375,0.345){${}_3$};
\node[right]at(0.5,1.47){${}^4$};
\node[right]at(1.375,0.595){${}^5$};
\node[right]at(2.75,0.72){${}_6$};
\node[above]at(2.925,1.125){${}_7$};
}
\gctwo[-1, -1,-1,3,12,-1,-1,-1, -1, -1,-1,5,-1,-1,-1,-1, -1, 15,14,-1,-1,-1,6,9]{3}{-2,0,-2}{2,-3,2}
\qquad\qquad
\setgcextra{
\node[right]at(0.125,0.095){${}_1$};
\node[right]at(2.25,1.72){${}^2$};
\node[right]at(1.625,0.345){${}^3$};
\node[right]at(2.5,1.47){${}^4$};
\node[right]at(2.625,0.595){${}_5$};
\node[right]at(0.75,1.22){${}^6$};
\node[above]at(2.925,1.125){${}_7$};
}
\gctwo[-1, 1,-1,-1,-1,-1,10,-1, -1, -1,-1,-1,-1,3,-1,-1, -1, -1,14,-1,12,5,-1,9]{3}{-2,0,-2}{2,-3,2}
}
$$
\caption[figPermWord]{The plots of permutation $1527634$
in $\begin{geommx}-1&\pos0&-1\\ \pos1&-1&\pos1 \end{geommx}$
%that is
%encoded by
associated with the
words $a_{32}a_{32}a_{11}a_{12}a_{21}a_{31}a_{32}$ and $a_{11}a_{32}a_{21}a_{32}a_{31}a_{12}a_{32}$.
Both plots correspond to the same gridding.
}\label{figPermWord}
\end{figure}
We are now in a position to describe the association between words and $M$-gridded permutations.
If $M$ is a \opmo{} matrix, then
we
let $\Sigma_M=\{a_{ij}:M_{i,j}\neq0\}$ be an alphabet of symbols, one for each nonzero cell in $M$.
If
we have a consistent orientation for $\Lambda_M$,
then
we can associate to each finite word $w_1\ldots w_n$ over $\Sigma_M$
a specific plot of
a permutation
in $\Geom(M)$
as follows: If $w_k=a_{ij}$, include the point at distance
%$\frac{k}{n+1}\sqrt{2}$
$k\+\sqrt{2}/(n+1)$
along line segment %$L_{i,j}\subseteq\Lambda_M$,
$L_{i,j}$
according to its orientation.
See Figure~\ref{figPermWord} for two examples.
Clearly, this induces a mapping %$\phi^\#$
from
the set of all finite words over $\Sigma_M$ to $\Geomhash(M)$.
In fact, it can readily be shown that this map is surjective, every $M$-gridded permutation corresponding to some word over $\Sigma_M$ (\cite{AABRV2011} Proposition~5.3).

As can be seen in Figure~\ref{figPermWord}, distinct words may be mapped to the same gridded permutation.
This occurs because the order in which two consecutive points are included is immaterial if they occur in cells that
are neither in the same column nor in the same row.
%share neither column nor row.
From the perspective of the words, adjacent symbols corresponding to such cells may be interchanged without changing the
gridded permutation.
This corresponds to a structure known as a %partially commutative
%monoid
%or
%a
{trace monoid}. %: %, which leads us to the following definition:

If we have a consistent orientation for standard figure $\Lambda_M$,
then
we define the \emph{trace monoid of~$M$},
which we denote by $\tmm$,
to be the set of equivalence classes of words over $\Sigma_M$ in which $a_{ij}$ and $a_{k\ell}$ commute (i.e.~$a_{ij}a_{k\ell}=a_{k\ell}a_{ij}$) whenever $i\neq k$ and $j\neq \ell$.
%We denote the trace monoid of $M$ by $\tmm$.
It is then relatively straightforward to show equivalence between gridded permutations and elements of the trace monoid:
\begin{lemma}[see \cite{AABRV2011} Proposition~7.1]\label{lemmaTraceMonoid}
If the standard figure $\Lambda_M$ has a consistent orientation, then gridded $n$-permutations in $\Geomhash(M)$ are in bijection with equivalence classes of words of length $n$ in
$\tmm$.
\end{lemma}
\vspace{-4pt}
Hence, by combining Lemmas~\ref{lemmaGRGriddings}, \ref{lemmaDoubleRefinementConsistentOrientation} and \ref{lemmaTraceMonoid} with Observation~\ref{obsGeomDouble}, we know that the growth rate of $\Geom(M)$ is equal to the growth rate of $\mathbb{M}(\drm)$ if it exists.
% Thus, $\gr(\Geom(M))=\gr(\mathbb{M}(\drm))$.
All that remains %for us to do
is to determine the growth rate of the trace monoid of a matrix.

% ================================================================
%\section{Proof of Theorem~\ref{thmGeomClassGrowthRate}} %: Enumerating trace monoids}

Trace monoids
%(also known as \emph{partially commutative monoids})
were first studied by Cartier \& Foata~\cite{CF1969}.
%They have proven to be particularly valuable in modelling concurrent computations; see, for example, the book edited by Diekert \& Rozenberg~\cite{DR1995}.
%They
Using extended M\"obius inversion, they
determined the general form of the generating function, as follows:
\begin{lemma}[\cite{CF1969};
%also Fisher~\cite{Fisher1989};
see also Flajolet \& Sedgewick~\cite{FS2009} Note~V.10] %, p.307]
The ordinary generating function for $\tmm$ is given by
%\begin{equation}\label{eqTraceMonoidGF}
$$
f_M(z) \;=\;
%\left(
\frac{1}{
\sum_{k\geqslant0} (-1)^k \+ r_k(M) \+ z^k
}
%\right)^{-1}
$$
%\end{equation}
where $r_k(M)$ is the number of $k$-subsets of $\Sigma_M$ whose elements commute pairwise.
\end{lemma}
Since symbols in $\tmm$ commute if and only if they correspond to cells that are neither in the same column nor in the same row,
it is easy to see that
%the numbers $r_k(M)$ are the \emph{rook numbers} for $M$:
$r_k(M)$ is the number of distinct ways of placing $k$ chess rooks on the nonzero entries of $M$ in such a way that no two rooks attack each other by being in the same column or row.
The numbers $r_k(M)$ are known as the \emph{rook numbers} for $M$ (see Riordan~\cite{Riordan2002}).
%(For more on rook numbers, see Kaplansky \& Riordan~\cite{KR1946} or Chapters~7 and~8 of Riordan~\cite{Riordan2002}.)
%Given that
Moreover, a matching
%(independent edge set)
in the row-column graph $G(M)$ also corresponds to a set of cells %,
no pair of which share a column or row.
So
%it is equally obvious that
the rook numbers for $M$ are
%simply
%just
%also
the same as
the numbers of matchings in $G(M)$: % the row-column graph of $M$:
\begin{obs}\label{obsRookNumIsMatchingNum}
For all $k\geqslant0$, $r_k(M)=m_k(G(M))$.
\end{obs}
Now, by elementary analytic combinatorics, we know that the growth rate of $\tmm$ is given by the reciprocal of the root
of the denominator %of~\eqref{eqTraceMonoidGF}.
of $f_M(z)$
that has least magnitude (see~\cite{FS2009} Theorem~IV.7).
The fact that this polynomial has a \emph{unique} root of smallest modulus
was proved by Goldwurm \& Santini in~\cite{GS2000}. It is real and positive by Pringsheim's Theorem. % (\cite{FS2009} Theorem IV.6).

But the \emph{reciprocal} of the \emph{smallest} root of a polynomial is the same as the \emph{largest} root of the \emph{reciprocal} polynomial (obtained by reversing the order of the coefficients). Hence,
%we have:
%\begin{lemma}\label{lemmaGRTraceMonoid}
%If
if $M$ has dimensions $t\ttimes u$
and $n=t+u$,
%is the number of vertices in $G(M)$,
then the growth rate of $\tmm$ is the largest (positive real) root of the polynomial
\begin{equation}\label{eqGMDefn}
%$$
g_M(z)
\;=\;
\frac{1}{
z^{\floor{n/2}} \+ f_M \!\left(\frac{1}{z}\right)\!
}
\;=\;
\sum_{k=0}^{\floor{n/2}} (-1)^k \+ r_k(M) \+ z^{\floor{n/2}-k}.
%$$
\end{equation}
%where $n=t+u$ is the number of vertices in $G(M)$.
%\end{lemma}
Here,
%Since $r_k(M)=0$ for $k>\floor{n/2}$,
$g_M(z)$ is the reciprocal polynomial of
%the denominator of $f_M(z)$
%$1/f_M(z)$
$\left(f_M(z)\right)^{-1}$
%$\frac{1}{f_M(z)}$
multiplied by some nonnegative power of $z$,
since $r_k(M)=0$ for all $k>\floor{n/2}$.
%indeed, $r_k(M)=0$ for $k>\min(t,u)$
Note also that $n$ is the number of vertices in $G(M)$.

If we now compare the definition of $g_M(z)$ in~\eqref{eqGMDefn}
%in Lemma~\ref{lemmaGRTraceMonoid}
with that of the matching polynomial $\mu_G(z)$ in~\eqref{eqMatchingPolynomialDefn}
and use Observation~\ref{obsRookNumIsMatchingNum}, then we see that:
%\begin{lemma}
%$g_M(z^2)=\mu_{G(M)}(z)$ if $n$ is even, and $g_M(z^2)=z^{-1}\+\mu_{G(M)}(z)$ if $n$ is odd.
$$
g_M(z^2) \;=\;
\begin{cases}
\phantom{z^{-1}\+}\mu_{G(M)}(z), & \text{if $n$ is even;} \\[3pt]
         z^{-1}\+ \mu_{G(M)}(z), & \text{if $n$ is odd.}
\end{cases}
$$
%\end{lemma}
Hence,
the largest root of $g_M(z)$ is the square of the largest root of $\mu_{G(M)}(z)$.

%This is
We now have
all we need to prove Theorem~\ref{thmGeomClassGrowthRate}: The growth rate of $\Geom(M)$ is equal to the growth rate of $\mathbb{M}(\drm)$ which equals the square of the largest root of $\mu_{G(\drm)}(z)$.

%\vspace{9pt}
In the above argument, we only employ the double
refinement $\drm$ to ensure that a consistent orientation is possible.
By
Lemma~\ref{lemmaNegativeCycles}, we know that if
$G(M)$
is free of negative cycles then $\Lambda_M$
can be consistently oriented.
Thus, we have the following special case of Theorem~\ref{thmGeomClassGrowthRate}:

\thmbox{
\begin{cor}\label{corNoNegCycles}
If
$G(M)$
contains no negative cycles,
then
the growth rate of
%geometric grid class
$\Geom(M)$ is
equal to
the square
of the largest root of
%the matching polynomial
$\mu_{G(M)}(z)$.
\end{cor}
} %\thmbox

\newpage % ***
% ================================================================
\section{Consequences}
\label{SectGeomImplications}
%\section{Discussion}
In this final section, we investigate some of the implications of Theorem~\ref{thmGeomClassGrowthRate}.
By considering properties of the matching polynomial, we characterise the growth rates of geometric grid classes in terms of the spectral radii of trees, prove a monotonicity result, and explore the influence of cycle parity on growth rates.
We then compare the growth rates of geometric grid classes with those of monotone grid classes.
Finally, we consider the effect
%impact
of subdividing edges in the row-column graph.

\newcommand{\deledge}{\!\setminus\!}
%\subsubsection*{Properties of the matching polynomial}
Let's begin by introducing some notation.
We denote the graph composed of two disjoint subgraphs $G$ and $H$ by $G+H$.
The graph resulting from deleting
from a graph $G$
the vertex $v$
(and all edges incident to $v$)
%(i.e.~the subgraph of $G$ induced by $V(G)\setminus\{v\}$)
is denoted $G-v$.
Generalising this,
if $H$ is a subgraph of~$G$, then $G-H$ is the graph obtained by deleting the vertices of $H$ from $G$.
In contrast, we use
$G\deledge e$ to denote
the graph resulting from deleting the edge $e$ from %the edge set of
$G$.
The number of connected components of $G$ is represented by $\mathrm{comp}(G)$.
The characteristic polynomial of a graph $G$ is denoted $\Phi_G(z)$. We use $\rho(G)$ to denote the spectral radius of $G$, the largest root of $\Phi_G(z)$.
Finally, we use $\lambda(G)$ for the largest root of the matching polynomial $\mu_G(z)$.

The matching polynomial was independently discovered a number of times, beginning with
Heilmann \& Lieb~\cite{HL1970} when investigating monomer-dimer systems in statistical physics.
It was first studied from a combinatorial perspective by Farrell~\cite{Farrell1979} and Gutman~\cite{Gutman1977}.
The theory was then further developed by Godsil \& Gutman~\cite{GG1981} and Godsil~\cite{Godsil1981}.
An introduction can be found in the books by Godsil~\cite{Godsil1993} and Lov\'asz \& Plummer~\cite{LP2009}.

The facts concerning the matching polynomial that we will use are covered by three lemmas.
As a consequence of the first, we only need to consider connected graphs:
\begin{lemma}[Farrell~\cite{Farrell1979}, Gutman~\cite{Gutman1977}]
\label{lemmaMuProductComponents}
The matching polynomial of a graph is the product of the matching polynomials of its connected components.
%$\mu_{G+ H}(z)=\mu_{G}(z)\+\mu_{H}(z)$
\end{lemma}
%\vspace{-9pt}
Thus, in particular:
\begin{cor}\label{corLambdaMaxComponents}
  $\lambda(G+H)=\max(\lambda(G),\lambda(H))$.
\end{cor}

%\vspace{9pt}
The second lemma relates the matching polynomial to the characteristic polynomial.
\begin{lemma}[Godsil \& Gutman~\cite{GG1981}] % Corollary~4.1]
\label{lemmaMuCycleDef}
If $\CCC_G$ consists of all nontrivial subgraphs of $G$ which are unions of vertex-disjoint cycles (i.e., all subgraphs of $G$ which are regular of degree 2), then
  $$
  \mu_G(z) \;=\; \Phi_G(z) \:+\: \sum_{C\in\CCC_G} \!\!2^{\+\mathrm{comp}(C)} \+ \Phi_{G-C}(z),
  $$
where $\Phi_{G-C}(z)=1$ if $C=G$.
\end{lemma}
%\vspace{-9pt}
As an immediate consequence, we have the following:
\begin{cor}[Sachs~\cite{Sachs1964}, Mowshowitz~\cite{Mowshowitz1972}, Lov\'asz \& Pelik\'an~\cite{LP1973}]
\label{corMuPhiTree}
The matching polynomial of a graph %, $\mu_G(z)$,
is identical to its characteristic polynomial %, $\Phi_G(z)$,
if and only if the graph is acyclic.
%If $F$ is a forest, $\mu_F(z)=\Phi_F(z)$.
\end{cor}
%\vspace{-9pt}
In particular, their largest roots are identical:
\begin{cor}\label{corLambdaRhoTree}
  If $G$ is a forest, %the roots of $\mu_G(z)$ coincide with the spectrum of $G$ and hence
  then
$\lambda(G)=\rho(G)$.
\end{cor}
%\vspace{-9pt}
Thus, using Corollaries~\ref{corNoNegCycles} and~\ref{corLambdaMaxComponents}, we have the following alternative characterisation for the growth rates of acyclic geometric grid classes:

\thmbox{
\begin{cor}\label{corForestRho}
  If $G(M)$ is a forest, then $\gr(\Geom(M))=\rho(G(M))^2$.
\end{cor}
} %\thmbox

%\vspace{12pt}
The last, and most important, of the three lemmas allows us to determine the largest root of the matching polynomial of a graph from
the spectral radius of a related tree.
%Lemma~\ref{lemmaExpandVertex}
It
is a consequence of the fact,
determined by Godsil in~\cite{Godsil1981}, % and~\cite{Godsil1993},
that the moments (sums of the powers of the roots) of
%the matching polynomial of a graph
$\mu_G(z)$ enumerate certain closed walks on~$G$, which he calls \emph{tree-like}.
This is analogous to the fact that the moments of
%the characteristic polynomial of $G$
$\Phi_G(z)$
count \emph{all} closed walks on~$G$.
On a tree, all closed walks are tree-like.

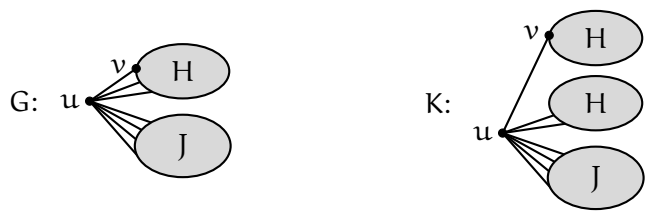
\begin{figure}[ht]
$$
\raisebox{12pt}{
\begin{tikzpicture}[scale=0.36]
    \draw [thick] (0,.3)--(1.7,1.5);
    \draw [thick] (0,.3)--(2.1,1.1);
    \draw [thick] (0,.3)--(3,0.75);
    \draw [fill=gray!30!white,thick] (3.4,1.4) circle [x radius=1.7, y radius=1];
    \draw [fill] (0,.3) circle [radius=0.15];
    \node[left]at(0,.3){$u$};
    \node[left]at(-1.5,.3){$G$:};
    \draw [fill] (1.7,1.5) circle [radius=0.15];
    \node[left]at(1.7,1.6){$v$};
    \node[]at(3.4,1.4){$H$};
    \draw [thick] (0,.3)--(3,-0.75);
    \draw [thick] (0,.3)--(2.33,-1);
    \draw [thick] (0,.3)--(2.17,-1.5);
    \draw [thick] (0,.3)--(2,-2);
    \draw [fill=gray!30!white,thick] (3.4,-1.35) circle [x radius=1.75, y radius=1.15];
    \node[]at(3.4,-1.35){$J$};
\end{tikzpicture}
}
\qquad\qquad\qquad
\begin{tikzpicture}[scale=0.36]
    \draw [fill=gray!30!white,thick] (3.4,3.8) circle [x radius=1.7, y radius=1];
    \draw [thick] (0,.3)--(1.7,3.9);
    \draw [thick] (0,.3)--(2,1);
    \draw [thick] (0,.3)--(3,0.75);
    \draw [fill=gray!30!white,thick] (3.4,1.4) circle [x radius=1.7, y radius=1];
    \draw [fill] (0,.3) circle [radius=0.15];
    \node[left]at(0,.3){$u$};
    \node[left]at(-1.5,1.4){$K$:};
    \draw [fill] (1.7,3.9) circle [radius=0.15];
    \node[left]at(1.7,4){$v$};
    \node[]at(3.4,1.4){$H$};
    \node[]at(3.4,3.8){$H$};
    \draw [thick] (0,.3)--(3,-0.75);
    \draw [thick] (0,.3)--(2.33,-1);
    \draw [thick] (0,.3)--(2.17,-1.5);
    \draw [thick] (0,.3)--(2,-2);
    \draw [fill=gray!30!white,thick] (3.4,-1.35) circle [x radius=1.75, y radius=1.15];
    \node[]at(3.4,-1.35){$J$};
\end{tikzpicture}
\qquad
$$
\caption{Expanding $G$ at $u$ along $uv$; $H$ is the component of $G-u$ that contains $v$}
\label{figExpandingVertex}
\end{figure}
\begin{lemma}[Godsil~\cite{Godsil1981}; see also~\cite{Godsil1993} and~\cite{LP2009}]
\label{lemmaExpandVertex}
  Let $G$ be a graph and let $u$ and $v$ be adjacent vertices in a cycle of $G$.
  Let $H$ be the component of $G-u$ that contains $v$.
  Now let $K$ be the graph constructed by taking a copy of $G\deledge uv$ and a copy of $H$ and joining the occurrence of $u$ in the copy of $G\deledge uv$ to the occurrence of $v$ in the copy of $H$ %.
  (see Figure~\ref{figExpandingVertex}).
  Then $\lambda(G)=\lambda(K)$.
\end{lemma}
%\vspace{-9pt}
The process that is described in Lemma~\ref{lemmaExpandVertex}
we will call
``\emph{expanding} $G$ \emph{at} $u$ \emph{along} $uv$''.
Each such expansion of a graph $G$ produces a graph with fewer cycles than $G$.
Repeated application of this process will thus eventually result in a forest $F$ such that $\lambda(F)=\lambda(G)$.
We shall say that $F$ results from \emph{fully expanding}~$G$.
Hence, by
Corollaries~\ref{corLambdaMaxComponents} and~\ref{corLambdaRhoTree},
the largest root of the matching polynomial of a graph equals the spectral radius of some tree:
for any graph~$G$, there is a tree $T$ such that $\lambda(G)=\rho(T)$.

It is readily observed that
every tree is the row-column graph of some geometric grid class. Thus
%As an immediate consequence,
we have the following characterisation of geometric grid class growth rates.

\thmbox{
\begin{cor}
  The set of growth rates of geometric grid classes consists of the squares of the spectral radii of trees.
\end{cor}
} %\thmbox

\HIDE{
In~\cite{Shearer1989}, Shearer proved that, for any $\rho\geqslant\sqrt{2+\sqrt{5}}$, there exists a sequence of trees $T_1,T_2,\ldots$ such that $\liminfty[k]\rho(T_k)=\rho$. Thus we have the following density result for the set of geometric grid class growth rates:

\thmbox{
\begin{cor}
  Every value at least $2+\sqrt{5}$ is a limit point of growth rates of
geometric grid classes.
\end{cor}
} %\thmbox
} %\HIDE

%\newpage
%\rulebreak
%\vspace{12pt}
The spectral radii of connected graphs satisfy the following strict monotonicity condition:
\begin{lemma}[\cite{CRS2010} Proposition~1.3.10]\label{lemmaStrictMono}
  If $G$ is connected and $H$ is a proper subgraph of $G$, then we have $\rho(H)<\rho(G)$.
\end{lemma}
%\vspace{-9pt}
Lemma~\ref{lemmaExpandVertex} enables us to prove
the analogous fact for the largest roots of matching polynomials, from which we can deduce a monotonicity result for geometric grid classes:
\begin{cor}\label{corStrictMono}%[see~\cite{Godsil1993} Corollary~6.1.3]
If $G$ is connected and $H$ is a proper subgraph of $G$, then $\lambda(H)<\lambda(G)$.
\end{cor}
%\vspace{-18pt}
\begin{proof}
  Suppose we fully expand $H$
  (at vertices $u_1,\ldots,u_k$, say),
  %until no cycles remain.
  then %, by Corollary~\ref{corLambdaRhoTree},
  the result is a forest $F$ such that $\lambda(H)=\rho(F)$.
  Now suppose that we repeatedly expand $G$ analogously at $u_1,\ldots,u_k$, and then continue to fully expand the resulting graph. The outcome is a tree $T$ (since $G$ is connected)
  such that $F$ is a proper subgraph of $T$ and
  $\lambda(G)=\rho(T)$.
  The result follows from Lemma~\ref{lemmaStrictMono}.
\end{proof}
%\vspace{-9pt}
Adding a non-zero cell to a \opmo{} matrix $M$ adds an edge to $G(M)$.
Thus, geometric grid classes satisfy the following monotonicity condition:

\thmbox{
\begin{cor}\label{corAddCell}
%Adding a non-zero cell to a connected geometric grid class while preserving connectivity increases its growth rate.
If $G(M)$ is connected and $M'$ results from adding a non-zero cell to $M$
in such a way that $G(M')$ is also connected,
%while preserving connectivity,
then $\gr(\Geom(M'))>\gr(\Geom(M))$.
\end{cor}
} %\thmbox

% --------------------------------------------------------
\subsection{Cycle parity}

The growth rate of a geometric grid class
depends on the parity of its cycles.
Consider the case of $G(M)$ being a cycle graph $C_n$.
If $G(M)$ is a negative cycle, then $G(\drm)=C_{2n}$.
Now, by Lemma~\ref{lemmaExpandVertex}, we have $\lambda(C_n)=\rho(P_{2n-1})$, where $P_n$ is the path graph on $n$ vertices.
The spectral radius of a graph on $n$ vertices is $2\cos\frac{\pi}{n+1}$.
So,
\begin{equation}\label{eqGRGeomCycle}
%$$
\gr(\Geom(M)) \;=\;
\begin{cases}
4\cos^2\frac{\pi}{2n}, & \text{if $G(M)$ is a positive cycle;} \\[5pt]
4\cos^2\frac{\pi}{4n}, & \text{if $G(M)$ is a negative cycle.}
\end{cases}
%$$
\end{equation}
Thus the geometric grid class whose row-column graph is a
negative cycle has a greater growth rate than the class
whose row-column graph is a positive cycle.
As another example,
\begin{equation}\label{eqGRGeomExNegCyc}
  \gr\!\left(\begin{geommx}1&\pos0&-1\\ 1&-1&\pos1 \end{geommx}\right) \;=\; 3+\sqrt{2} \;\approx\; 4.41421,
\end{equation}
%Figure~\ref{figGeomClass}
whereas
%the growth rate of
\begin{equation}\label{eqGRGeomExPosCyc}
  \gr\!\left(\begin{geommx}-1&\pos0&-1\\ \pos1&-1&\pos1 \end{geommx}\right) \;=\; 4.
\end{equation}
The former, containing a negative cycle,
has a greater growth rate than the latter, whose cycle is positive.
This is typical;
we will prove the following result:
%determine sufficient conditions under which the growth rate is increased by making positive cycles negative.

\thmbox{
\begin{cor}\label{corNegateCell}
If $G(M)$ is connected and contains no negative cycles, and $M_{\+\negsub}$ %$G(M_{\+\negsub})$
results from %negating
changing the sign of
a single entry of $M$ that is in a cycle (thus making one or more cycles in $G(M_{\+\negsub})$ negative),
then
$\gr(\Geom(M_{\+\negsub})) > \gr(\Geom(M))$.
\end{cor}
} %\thmbox

In order to do this, we need to consider the structure of $G(\drm)$.
The graph $G(\drm)$ can be constructed from $G(M)$ as follows: % see GeomNew p.1
If $G(M)$ has vertex set $\{v_1,\ldots,v_n\}$, then we let $G(\drm)$ have vertices $v_1,\ldots,v_n$ and $v'_1,\ldots,v'_n$.
If $v_iv_j$ is a positive edge in $G(M)$, then in $G(\drm)$ we add an edge between $v_i$ and $v_j$ and also between $v'_i$ and $v'_j$.
On the other hand, if $v_iv_j$ is a negative edge in $G(M)$, then in $G(\drm)$ we join $v_i$ to $v'_j$ and $v'_i$ to $v_j$.
The correctness of this construction follows directly from the definitions of double refinement and of the row-column graph of a matrix.
%Note that the map which takes $v_i$ to $v'_i$ and $v'_i$ to $v_i$ is an automorphism of $G(\drm)$.
For an illustration, compare the graph in Figure~\ref{figDoubleRefinement} against that in Figure~\ref{figGeomClass}.

Note that if $v_1,\ldots,v_k$ is an even
$k$-cycle in $G(M)$, then $G(\drm)$ contains
two vertex-disjoint even $k$-cycles, the union of whose vertices is $\{v_1,\ldots,v_k,v'_1,\ldots,v'_k\}$.
In contrast, if $v_1,\ldots,v_\ell$ is an odd
$\ell$-cycle in $G(M)$, then $G(\drm)$ contains a $2\+\ell$-cycle on $\{v_1,\ldots,v_\ell,v'_1,\ldots,v'_\ell\}$ in which $v_i$ is opposite $v'_i$ (i.e. $v'_i$ is at distance $\ell$ from $v_i$ around the cycle) for each $i$, $1\leqslant i\leqslant\ell$.
We make the following additional observations:
\begin{obs}\label{obsEvenx2} % GeomNew p.7
  If $G(M)$ has no odd cycles, then $G(\drm)=G(M)+ G(M)$.
\end{obs}
\begin{obs}\label{obsOddx2}
  If $G(M)$ is connected and has an odd cycle, then $G(\drm)$ is connected.
\end{obs}
%\vspace{-9pt}
We now have all we require to prove our cycle parity result.

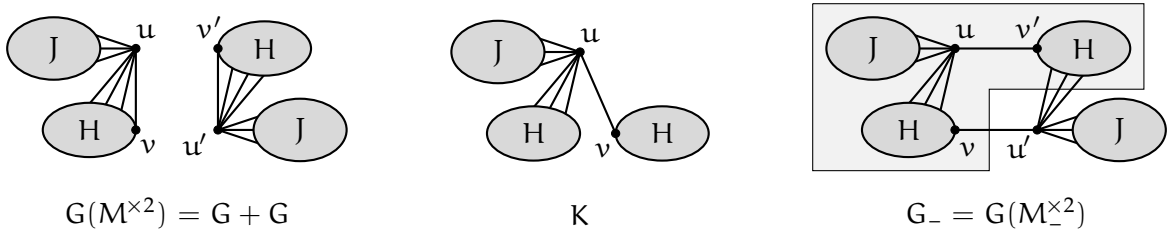
\begin{figure}[ht] %[ht]
$$
\begin{tikzpicture}[scale=0.36]
    \draw [thick] (0,3)--(-2.5,0);
    \draw [thick] (0,3)--(-1.6,0);
    \draw [thick] (0,3)--(-0.7,0);
    \draw [thick] (0,3)--(-3,4);
    \draw [thick] (0,3)--(-3,3);
    \draw [thick] (0,3)--(-3,2);
    \draw [fill=gray!30!white,thick] (-1.7,0) circle [x radius=1.7, y radius=1];
    \node[]at(-1.7,0){$H$};
    \draw [fill=gray!30!white,thick] (-3,3) circle [x radius=1.7, y radius=1.15];
    \node[]at(-3,3){$J$};
    \draw [thick] (0,0)--(0,3);
    \draw [fill] (0,3) circle [radius=0.15];
    \draw [fill] (0,0) circle [radius=0.15];
    \node[above right]at(0,3){$\!\!u$};
    \node[below right]at(0,0){$\!v$};
    \draw [thick] (3,0)--(3.7,3);
    \draw [thick] (3,0)--(4.6,3);
    \draw [thick] (3,0)--(5.5,3);
    \draw [thick] (3,0)--(6,1);
    \draw [thick] (3,0)--(6,0);
    \draw [thick] (3,0)--(6,-1);
    \draw [fill=gray!30!white,thick] (4.7,3) circle [x radius=1.7, y radius=1];
    \node[]at(4.7,3){$H$};
    \draw [fill=gray!30!white,thick] (6,0) circle [x radius=1.7, y radius=1.15];
    \node[]at(6,0){$J$};
    \draw [thick] (3,0)--(3,3);
    \draw [fill] (3,3) circle [radius=0.15];
    \draw [fill] (3,0) circle [radius=0.15];
    \node[above left]at(3,3){$v'\!\!\!$};
    \node[below left]at(3,.27){$u\!'$};
    \node[]at(1.5,-3.1){$G(\drm)\+=\+G+ G$};
\end{tikzpicture}
\qquad\quad\;\:
\begin{tikzpicture}[scale=0.36]
    \draw [thick] (0,3)--(-2.5,0);
    \draw [thick] (0,3)--(-1.6,0);
    \draw [thick] (0,3)--(-0.7,0);
    \draw [thick] (0,3)--(-3,4);
    \draw [thick] (0,3)--(-3,3);
    \draw [thick] (0,3)--(-3,2);
    \draw [fill=gray!30!white,thick] (-1.7,0) circle [x radius=1.7, y radius=1];
    \node[]at(-1.7,0){$H$};
    \draw [fill=gray!30!white,thick] (3,0) circle [x radius=1.7, y radius=1];
    \node[]at(3,0){$H$};
    \draw [fill=gray!30!white,thick] (-3,3) circle [x radius=1.7, y radius=1.15];
    \node[]at(-3,3){$J$};
    \draw [thick] (0,3)--(1.3,0);
    \draw [fill] (0,3) circle [radius=0.15];
    \draw [fill] (1.3,0) circle [radius=0.15];
    \node[above right]at(0,3){$\!\!u$};
    \node[below left]at(1.3,0){$v\!$};
    \node[]at(0,-3.1){$\phantom{G_1}K\phantom{G_1}$};
\end{tikzpicture}
\qquad\quad\;\:
\begin{tikzpicture}[scale=0.36]
    \draw [fill=gray!10!white] (-5.2,4.65)--(6.9,4.65)--(6.9,1.5)--(1.25,1.5)--(1.25,-1.5)--(-5.2,-1.5)--(-5.2,4.65);
    \draw [thick] (0,3)--(-2.5,0);
    \draw [thick] (0,3)--(-1.6,0);
    \draw [thick] (0,3)--(-0.7,0);
    \draw [thick] (0,3)--(-3,4);
    \draw [thick] (0,3)--(-3,3);
    \draw [thick] (0,3)--(-3,2);
    \draw [fill=gray!30!white,thick] (-1.7,0) circle [x radius=1.7, y radius=1];
    \node[]at(-1.7,0){$H$};
    \draw [fill=gray!30!white,thick] (-3,3) circle [x radius=1.7, y radius=1.15];
    \node[]at(-3,3){$J$};
    \draw [thick] (0,0)--(3,0);
    \draw [fill] (0,3) circle [radius=0.15];
    \draw [fill] (0,0) circle [radius=0.15];
    \node[above right]at(0,3){$\!\!u$};
    \node[below right]at(0,-0.1){$\!v$};
    \draw [thick] (3,0)--(3.7,3);
    \draw [thick] (3,0)--(4.6,3);
    \draw [thick] (3,0)--(5.5,3);
    \draw [thick] (3,0)--(6,1);
    \draw [thick] (3,0)--(6,0);
    \draw [thick] (3,0)--(6,-1);
    \draw [fill=gray!30!white,thick] (4.7,3) circle [x radius=1.7, y radius=1];
    \node[]at(4.7,3){$H$};
    \draw [fill=gray!30!white,thick] (6,0) circle [x radius=1.7, y radius=1.15];
    \node[]at(6,0){$J$};
    \draw [thick] (0,3)--(3,3);
    \draw [fill] (3,3) circle [radius=0.15];
    \draw [fill] (3,0) circle [radius=0.15];
    \node[above left]at(3,3){$v'\!\!\!$};
    \node[below left]at(3,.17){$u\!'\!$};
    \node[]at(1.5,-3.1){$G_{\+\negsub}\+=\+G(M_{\+\negsub}^{\!\times2})$};
\end{tikzpicture}
\vspace{-10.5pt}
$$
\caption{Graphs used in the proof of Corollary~\ref{corNegateCell}}
\label{figNegCycle}
\end{figure}
%\vspace{-9pt}
\begin{proof}[Proof of Corollary~\ref{corNegateCell}]
Let $G=G(M)$ and $G_{\+\negsub}=G(M_{\+\negsub}^{\!\times2})$, and let
$uv$ be the edge in $G$ corresponding to the entry in $M$ that is negated to create $M_{\+\negsub}$.
Since $G$ %is connected and
contains no negative cycles, by Observation~\ref{obsEvenx2},
$G(\drm)=G+G$.
%is the disjoint union of two copies of $G$.
Thus,
since $G$ is connected, it has the form at the left of Figure~\ref{figNegCycle}, in which $H$ is the component of $G-u$ containing $v$.
Moreover,
we have
$\gr(\Geom(M))=\lambda(G)$.
(This also follows from Corollary~\ref{corNoNegCycles}.)
Now, if we expand $G$ at $u$ along $uv$, by Lemma~\ref{lemmaExpandVertex}, %we have
$\lambda(G)=\lambda(K)$, where $K$ is the graph in the centre of Figure~\ref{figNegCycle}.

On the other hand, $G_{\+\negsub}$ is obtained from $G(\drm)$ by removing the edges $uv$ and $u'v'$, and adding $uv'$ and $u'v$, as shown at the
right of Figure~\ref{figNegCycle}.
It is readily observed that $K$ is a proper subgraph of $G_{\+\negsub}$ (see the shaded box in Figure~\ref{figNegCycle}), and hence, by Corollary~\ref{corStrictMono}, $\lambda(K)<\lambda(G_{\+\negsub})$. Since $\gr(\Geom(M_{\+\negsub}))=\lambda(G_{\+\negsub})$, the result follows.
\end{proof}
%It seems possible
Thus, making the first negative cycle increases the growth rate.
We suspect, in fact,
that the following stronger statement is also true:

\thmbox{
\begin{conj}\label{conjNegateCell}
If $G(M)$ is connected and $M_{\+\negsub}$ results from negating a single entry of $M$
that is in one or more positive cycles but in no negative cycle,
%(thus making one or more cycles in $G(M_{\+\negsub})$ negative),
then
$\gr(\Geom(M_{\+\negsub})) > \gr(\Geom(M))$.
\end{conj}
} %\thmbox

To prove this more general result seems to require some new ideas.
If $G(M)$ already contains a negative cycle, then $G(\drm)$ is connected, and, when this is the case,
there appears to be no obvious way to generate a subgraph of $G(M_{\+\negsub}^{\!\times2})$
by expanding $G(\drm)$.

% --------------------------------------------------------
\subsection{Monotone grid classes}

In a recent paper~\cite{Bevan2013}, we established the growth rates of \emph{monotone} grid classes. If $M$ is a \opmo{} matrix, then the \emph{monotone grid class} $\Grid(M)$ consists of those permutations that can be plotted as a subset of some figure consisting of the union of \emph{any monotonic curves} $\Gamma_{i,j}$ with the same endpoints as the $L_{i,j}$ in $\Lambda_M$.
This permits greater flexibility in the positioning of points in the cells, so
$\Geom(M)$ is a subset of $\Grid(M)$
and we have $\gr(\Geom(M))\leqslant\gr(\Grid(M))$.
In fact, the geometric grid class
$\Geom(M)$ and the monotone grid class $\Grid(M)$
are identical
if
and only if
$G(M)$
is acyclic %, then
(Theorem~3.2 in~\cite{AABRV2011}).
%if and only if $G(M)$
%is acyclic.
%has no cycles.
Hence, if $G(M)$ is a forest, %then
%we have
$\gr(\Geom(M))=\gr(\Grid(M))$.
We determined
in~\cite{Bevan2013}
that
the growth rate of monotone grid class $\Grid(M)$ is equal to the square of the spectral radius of $G(M)$.
%$\gr(\Grid(M))=\rho(G(M))^2$.
For acyclic $G(M)$, this is consistent with
the growth rate of the geometric grid class as given by
Corollary~\ref{corForestRho}.
%the fact that $\lambda(G(M))=\rho(G(M))$.

Typically,
the growth rate of a monotone grid class will be greater than
%the growth rate
that
of the corresponding geometric grid class.
For example,
if $G(M)$ is a cycle then $\gr(\Grid(M))=4$, whereas from~\eqref{eqGRGeomCycle} we have $\gr(\Geom(M))<4$.
And we have
$$
\gr\!\left(\begin{gridmx}1&\pos0&-1\\ 1&-1&\pos1 \end{gridmx}\right)
\;=\;
\gr\!\left(\begin{gridmx}-1&\pos0&-1\\ \pos1&-1&\pos1 \end{gridmx}\right)
\;=\;
\thalf(5+\sqrt{17})
\;\approx\;
4.56155,
$$
which should be compared with~\eqref{eqGRGeomExNegCyc} and~\eqref{eqGRGeomExPosCyc}.

The fact that the growth rate of the monotone grid class is strictly greater
is a consequence of the fact that, if $G$ is connected and not acyclic, then %there is a strict inequality between
$\lambda(G)$ and $\rho(G)$ are distinct:
\begin{lemma}[Godsil \& Gutman~\cite{GG1981}] % Theorem~6]
\label{lemmaLambdaLessRhoCyclic}
If %a graph
$G$ is connected and contains a cycle, then
$\lambda(G)<\rho(G)$.
%the spectral radius of $G$ strictly exceeds the largest root of the matching polynomial $\mu_G(z)$.
\end{lemma}
%\vspace{-18pt}
\begin{proof} % GeomNew p.2
By Lemma~\ref{lemmaStrictMono},
if $C$ is a nonempty subgraph of $G$, then $\rho(G-C)<\rho(G)$.
So we have $\Phi_{G-C}(z)>0$ for all $z\geqslant\rho(G)$.
Moreover, $\Phi_G(z)\geqslant0$ for $z\geqslant\rho(G)$. So,
since $G$ contains a cycle, from Lemma~\ref{lemmaMuCycleDef} we can deduce that $\mu_G(z)>0$ if $z\geqslant\rho(G)$, and thus
$\lambda(G)<\rho(G)$.
\end{proof}
%\vspace{-9pt}
Note that,
analogously to Observation~\ref{obsGeomDouble},
$\Grid(\drm)=\Grid(M)$.
Hence it must be the case that $\rho(G(\drm))=\rho(G(M))$, the growth rate of a monotone grid class thus being independent of the parity of its cycles.
As a consequence, from Lemma~\ref{lemmaLambdaLessRhoCyclic} we can deduce that in the non-acyclic case
there is a strict inequality between the growth rate of a geometric grid class and the growth rate of the corresponding monotone grid class:

\thmbox{
\begin{cor}
  If $G(M)$ is connected,
  % and not acyclic,
  then $\gr(\Geom(M))<\gr(\Grid(M))$ %.
  if and only if $G(M)$ contains a cycle.
\end{cor}
} %\thmbox

% --------------------------------------------------------
\subsection{Subdivision of edges}

One surprising result in~\cite{Bevan2013} concerning the growth rates of monotone grid classes is the fact that classes
whose row-column graphs
have longer internal paths or cycles exhibit \emph{lower} growth rates.
An edge $e$ of a graph $G$ is said to lie on an \emph{endpath} of $G$ if $G\deledge e$ is disconnected and one of its components is a (possibly trivial) path. An edge that does not lie on an endpath is said to be \emph{internal}.
The following result of Hoffman \& Smith states that the subdivision of an edge increases or decreases the spectral radius of the graph depending on whether the edge lies on an endpath or is internal:
\begin{lemma}[Hoffman \& Smith~\cite{HS1975}]\label{lemmaRhoSubdivision}
  Let $G$ be a connected graph and $G'$ be obtained from $G$ by subdividing an edge $e$. If $e$ lies on an endpath, then $\rho(G')>\rho(G)$. Otherwise (if $e$ is an internal edge), $\rho(G')\leqslant\rho(G)$, with equality if and only if $G$ is a cycle or has the following form (which we call an ``$H$~graph''):
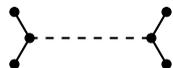
$$
  \begin{tikzpicture}[scale=0.4]
    %\draw [thick] (0,0)--(1,0);
    \draw [thick,dashed] (0,0)--(4,0);
    %\draw [thick] (5,0)--(6,0);
    \draw [thick] (-.5,.866)--(0,0)--(-.5,-.866);
    \draw [thick] (4.5,.866)--(4,0)--(4.5,-.866);
    %\foreach \x in {0,...,6}
    %  \draw [fill] (\x,0) circle [radius=0.15];
    \draw [fill] (0,0) circle [radius=0.15];
    %\draw [fill] (1,0) circle [radius=0.15];
    %\draw [fill] (5,0) circle [radius=0.15];
    \draw [fill] (4,0) circle [radius=0.15];
    \draw [fill] (-.5,.866) circle [radius=0.15];
    \draw [fill] (-.5,-.866) circle [radius=0.15];
    \draw [fill] (4.5,.866) circle [radius=0.15];
    \draw [fill] (4.5,-.866) circle [radius=0.15];
  \end{tikzpicture}
$$
\end{lemma}
Thus for monotone grid classes, if $G(M)$ is connected,
and $G(M')$ is obtained from $G(M)$ by the subdivision of one or more internal edges, then $\gr(\Grid(M'))\leqslant\gr(\Grid(M))$.
%unless $G(M)$ is a cycle or an $H$ graph. % (see Bevan~\cite{Bevan2013}).

As we will see, the situation is not as simple for geometric grid classes.
The effect of edge subdivision on the largest root of the matching polynomial does not seem to have been addressed previously. 
In fact, the subdivision of an edge that is in a cycle may
cause $\lambda(G)$ %the largest root of the matching polynomial 
to increase or decrease, or may leave it unchanged.
See Figures~\ref{figGeomSubdivisionIncr}--\ref{figGeomSubdivisionDecr} for illustrations of the three cases.
We investigate this further below.
However, if the edge being subdivided is not on a cycle in $G$, then the behaviour of $\lambda(G)$ mirrors that of $\rho(G)$, as we now demonstrate:
\begin{lemma}\label{lemmaSubdividing1}
  Let $G$ be a connected graph and $G'$ be obtained from $G$ by subdividing an edge $e$.
  If $e$ lies on an endpath, then $\lambda(G')>\lambda(G)$.
  However, if $e$ is an internal edge and not on a cycle, then $\lambda(G')\leqslant\lambda(G)$, with equality if and only if $G$ is an $H$ graph.
\end{lemma}

\begin{figure}[ht]
$$
\raisebox{12pt}{
\begin{tikzpicture}[scale=0.36]
    \draw [thick] (0,.3)--(1.7,1.5);
    \draw [thick] (0,.3)--(2.1,1.1);
    \draw [thick] (0,.3)--(3,0.75);
    \draw [thick,fill=gray!30!white] (3,1.4) circle [x radius=1.3, y radius=1];
    \draw [thick,fill=gray!30!white] (7.2,1.4) circle [x radius=1.2, y radius=.9];
    \draw [fill] (0,.3) circle [radius=0.15];
    \node[left]at(0,.3){$u$};
    \node[left]at(-1.5,.3){$G$:};
    \draw [fill] (1.7,1.5) circle [radius=0.15];
    %\draw [thick,fill=gray!30!white] (5.5,1.4)--(6.77,0.98)--(6.77,1.82)--(5.5,1.4);
    \draw [fill] (4.3,1.4) circle [radius=0.15];
    \draw [fill] (6,1.4) circle [radius=0.15];
    \draw [ultra thick] (4.3,1.4)--(6,1.4);
    \node[above]at(5.15,1.4){$e$};
    \node[left]at(1.7,1.6){$v$};
    \node[]at(3,1.4){$\,H_1$};
    \node[]at(7.2,1.4){$H_2$};
    \draw [thick] (0,.3)--(3,-0.75);
    \draw [thick] (0,.3)--(2.33,-1);
    \draw [thick] (0,.3)--(2.17,-1.5);
    \draw [thick] (0,.3)--(2,-2);
    \draw [thick,fill=gray!30!white] (3.4,-1.35) circle [x radius=1.75, y radius=1.15];
    \node[]at(3.4,-1.35){$J$};
\end{tikzpicture}
}
\qquad\qquad\qquad
\begin{tikzpicture}[scale=0.36]
    \draw [thick,fill=gray!30!white] (3,3.8) circle [x radius=1.3, y radius=1];
    \draw [thick] (0,.3)--(1.7,3.9);
    \draw [thick] (0,.3)--(2.1,1.1);
    \draw [thick] (0,.3)--(3,0.75);
    \draw [thick,fill=gray!30!white] (3,1.4) circle [x radius=1.3, y radius=1];
    \draw [fill] (0,.3) circle [radius=0.15];
    \node[left]at(0,.3){$u$};
    \node[left]at(-1.5,1.4){$K$:};
    \draw [fill] (1.7,3.9) circle [radius=0.15];
    %
    %\draw [thick,fill=gray!30!white] (5.5,3.8)--(6.77,3.38)--(6.77,4.22)--(5.5,3.8);
    \draw [thick,fill=gray!30!white] (7.2,3.8) circle [x radius=1.2, y radius=.9];
    \draw [fill] (4.3,3.8) circle [radius=0.15];
    \draw [fill] (6,3.8) circle [radius=0.15];
    \draw [ultra thick] (4.3,3.8)--(6,3.8);
    \node[above]at(5.15,3.8){$e_1$};
    %
    %\draw [thick,fill=gray!30!white] (5.5,1.4)--(6.77,0.98)--(6.77,1.82)--(5.5,1.4);
    \draw [thick,fill=gray!30!white] (7.2,1.4) circle [x radius=1.2, y radius=.9];
    \draw [fill] (4.3,1.4) circle [radius=0.15];
    \draw [fill] (6,1.4) circle [radius=0.15];
    \draw [ultra thick] (4.3,1.4)--(6,1.4);
    \node[above]at(5.15,1.4){$e_2$};
    \node[left]at(1.7,4){$v$};
    \node[]at(3,1.4){$\,H_1$};
    \node[]at(3,3.8){$\,H_1$};
    \node[]at(7.2,1.4){$H_2$};
    \node[]at(7.2,3.8){$H_2$};
    \draw [thick] (0,.3)--(3,-0.75);
    \draw [thick] (0,.3)--(2.33,-1);
    \draw [thick] (0,.3)--(2.17,-1.5);
    \draw [thick] (0,.3)--(2,-2);
    \draw [thick,fill=gray!30!white] (3.4,-1.35) circle [x radius=1.75, y radius=1.15];
    \node[]at(3.4,-1.35){$J$};
\end{tikzpicture}
\qquad
$$
\caption{Graphs used in the proof of Lemma~\ref{lemmaSubdividing1}} %; the triangles represent non-path graphs}
\label{figExpandingVertexEInternal}
\end{figure}
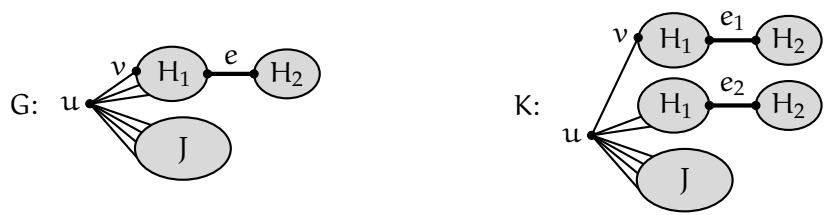
\begin{proof}
If $e$ lies on an endpath, then $G$ is a proper subgraph of $G'$ and so
the result follows from Corollary~\ref{corStrictMono}.
On the other hand, if $e$ is internal and $G$ is acyclic, the conclusion is a consequence of Corollary~\ref{corLambdaRhoTree} and Lemma~\ref{lemmaRhoSubdivision}.
Thus, we need only consider the situation in which
$e$ is internal and $G$ contains a cycle. We proceed by induction on the number of cycles in $G$, acyclic graphs constituting the base case.
Let $uv$ be an edge in a cycle of $G$ such that $u$ is not an endvertex of~$e$.
Now, let $K$ be the result of expanding $G$ at $u$ along $uv$, and let $K'$, analogously, be the result of expanding $G'$ at $u$ along~$uv$.
%$K$ has fewer cycles than $G$.

We consider the effect of the expansion of $G$ upon $e$ and the effect of the expansion of $G'$ upon the two edges resulting from the subdivision of $e$.
If $e$ is in the component of $G-u$ containing~$v$, then $e$ is duplicated in $K$, both copies of $e$ remaining internal (see Figure~\ref{figExpandingVertexEInternal}).
Moreover, $K'$ results from subdividing both copies of $e$ in $K$.
Conversely, if $e$ is in a component of $G-u$ not containing~$v$, then $e$ is not duplicated in $K$ (and remains internal).
In this case, $K'$ results from subdividing $e$ in $K$.
In either case, $K'$ is the result of subdividing internal edges of $K$ (a graph with fewer cycles than $G$), and so
the result follows from the induction hypothesis.
\end{proof}
Now,
the subdivision of an edge
of a row-column graph
that is not on a cycle has no effect on the parity of the cycles.
Hence,
we
have the following conclusion
for the growth rates of geometric grid classes:

\thmbox{
\begin{cor}
If $G(M)$ is connected,
and $G(M')$ is obtained from $G(M)$ by the subdivision of one or more internal edges not on a cycle, then $\gr(\Geom(M'))\leqslant\gr(\Geom(M))$, with equality if and only if $G(M)$ is an $H$ graph.
\end{cor}
} %\thmbox

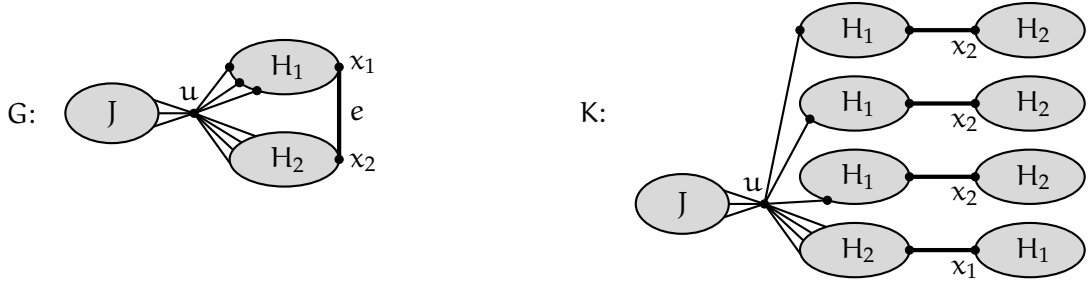
\begin{figure}[ht] %[ht]
\vspace{6pt}
$$
\raisebox{34.2pt}{
\begin{tikzpicture}[scale=0.36]
    \draw [thick] (0,3)--(-3,4);
    \draw [thick] (0,3)--(-3,3);
    \draw [thick] (0,3)--(-3,2);
    \draw [thick] (0,3)--(1.3,4.7);
    \draw [thick] (0,3)--(1.662,4.126);
    \draw [thick] (0,3)--(2.3,3.834);
    \draw [thick] (0,3)--(2.4,2.1);
    \draw [thick] (0,3)--(2.1,1.7);
    \draw [thick] (0,3)--(1.8,1.3);
    \draw [thick] (0,3)--(1.5,0.9);
    \draw [ultra thick] (5.3,1.3)--(5.3,4.7);
    \draw [fill=gray!30!white,thick] (3.3,4.7) circle [x radius=2, y radius=1];
    \node[]at(3.3,4.7){$\phantom{{}_.}H_1$};
    \draw [fill=gray!30!white,thick] (3.3,1.3) circle [x radius=2, y radius=1];
    \node[]at(3.3,1.3){$\phantom{{}_.}H_2$};
    \draw [fill=gray!30!white,thick] (-3,3) circle [x radius=1.7, y radius=1.1];
    \node[]at(-3,3){$J$};
    \draw [fill] (0,3) circle [radius=0.15];
    \draw [fill] (1.3,4.7) circle [radius=0.15];
    \draw [fill] (1.662,4.126) circle [radius=0.15];
    \draw [fill] (2.3,3.834) circle [radius=0.15];
    \draw [fill] (5.3,4.7) circle [radius=0.15];
    \draw [fill] (5.3,1.3) circle [radius=0.15];
    \node[above left]at(0,3.1){$u\!\!\!$};
    %\node[above left]at(1.3,4.6){$v\!$};
    \node[right]at(5.3,3){$e$};
    \node[right]at(5.3,4.8){$x_1$};
    \node[right]at(5.3,1.2){$x_2$};
    \node[left]at(-5.4,3){$G$:};
    \node[below left]at(0,1.3){$\phantom{x_1}$};
\end{tikzpicture}
}
\qquad\qquad\qquad
\begin{tikzpicture}[scale=0.36]
    \draw [thick] (0,3)--(1.3,9.4);
    \draw [ultra thick] (5.3,9.4)--(7.7,9.4);
    \draw [fill=gray!30!white,thick] (3.3,9.4) circle [x radius=2, y radius=1];
    \node[]at(3.3,9.4){$\phantom{{}_.}H_1$};
    \draw [fill=gray!30!white,thick] (9.7,9.4) circle [x radius=2, y radius=1];
    \node[]at(9.7,9.4){$\phantom{{}_.}H_2$};
    \draw [fill] (1.3,9.4) circle [radius=0.15];
    \draw [fill] (5.3,9.4) circle [radius=0.15];
    \draw [fill] (7.7,9.4) circle [radius=0.15];
    \node[below left]at(8.2,9.4){$x_2$};
    \draw [thick] (0,3)--(1.662,6.126);
    \draw [ultra thick] (5.3,6.7)--(7.7,6.7);
    \draw [fill=gray!30!white,thick] (3.3,6.7) circle [x radius=2, y radius=1];
    \node[]at(3.3,6.7){$\phantom{{}_.}H_1$};
    \draw [fill=gray!30!white,thick] (9.7,6.7) circle [x radius=2, y radius=1];
    \node[]at(9.7,6.7){$\phantom{{}_.}H_2$};
    \draw [fill] (1.662,6.126) circle [radius=0.15];
    \draw [fill] (5.3,6.7) circle [radius=0.15];
    \draw [fill] (7.7,6.7) circle [radius=0.15];
    \node[below left]at(8.2,6.7){$x_2$};
    \draw [thick] (0,3)--(2.3,3.134);
    \draw [ultra thick] (5.3,4)--(7.7,4);
    \draw [fill=gray!30!white,thick] (3.3,4) circle [x radius=2, y radius=1];
    \node[]at(3.3,4){$\phantom{{}_.}H_1$};
    \draw [fill=gray!30!white,thick] (9.7,4) circle [x radius=2, y radius=1];
    \node[]at(9.7,4){$\phantom{{}_.}H_2$};
    \draw [fill] (2.3,3.134) circle [radius=0.15];
    \draw [fill] (5.3,4) circle [radius=0.15];
    \draw [fill] (7.7,4) circle [radius=0.15];
    \node[below left]at(8.2,4){$x_2$};
    \draw [thick] (0,3)--(-3,4);
    \draw [thick] (0,3)--(-3,3);
    \draw [thick] (0,3)--(-3,2);    
    \draw [thick] (0,3)--(2.4,2.1);
    \draw [thick] (0,3)--(2.1,1.7);
    \draw [thick] (0,3)--(1.8,1.3);
    \draw [thick] (0,3)--(1.5,0.9);
    \draw [ultra thick] (5.3,1.3)--(7.7,1.3);
    \draw [fill=gray!30!white,thick] (3.3,1.3) circle [x radius=2, y radius=1];
    \node[]at(3.3,1.3){$\phantom{{}_.}H_2$};
    \draw [fill=gray!30!white,thick] (9.7,1.3) circle [x radius=2, y radius=1];
    \node[]at(9.7,1.3){$\phantom{{}_.}H_1$};
    \draw [fill=gray!30!white,thick] (-3,3) circle [x radius=1.7, y radius=1.1];
    \node[]at(-3,3){$J$};
    \draw [fill] (0,3) circle [radius=0.15];
    \draw [fill] (5.3,1.3) circle [radius=0.15];
    \draw [fill] (7.7,1.3) circle [radius=0.15];
    \node[above left]at(0,3.1){$u\!\:\!\!$};
    %\node[above left]at(1.3,4.6){$v\!$};
    %\node[above]at(6.5,4.7){$e_1$};
    %\node[above]at(6.5,1.3){$e_2$};
    \node[below left]at(8.2,1.3){$x_1$};
    \node[left]at(-5.4,6.4){$K$:};
\end{tikzpicture}
$$
\caption{Graphs used in Lemma~\ref{lemmaSubdividing2}}
\label{figExpandingCycle}
\end{figure}
Let us now investigate the effect of subdividing an edge $e$ that lies on a cycle.
We restrict our attention to graphs 
in which there is a vertex $u$ such that 
the two endvertices of $e$ are in distinct
components of $(G\deledge e) - u$.
See the graph at the left of Figure~\ref{figExpandingCycle} for an illustration.
We leave the consideration of multiply-connected graphs that fail to satisfy this condition for future study.
\begin{lemma}\label{lemmaSubdividing2}
Let $G$ be a connected graph and $e=x_1x_2$ an edge on a cycle $C$ of $G$.
Let $u$ be a vertex on $C$, and let $H_1$ and $H_2$ be the distinct
components of $(G\deledge e) - u$ that contain $x_1$ and $x_2$ respectively.
Finally, let $G'$ be the graph obtained from $G$ by subdividing $e$.
  \vspace{-9pt}
  \begin{enumerate}[(a)]
    \item If, for $i\in\{1,2\}$, $H_i$ is a (possibly trivial) path of which $x_i$ is an endvertex, then $\lambda(G')>\lambda(G)$.
    \item If, for $i\in\{1,2\}$, $H_i$ is not a path or is a path of which $x_i$ is not an endvertex, then $\lambda(G')<\lambda(G)$.
  \end{enumerate}
  \vspace{-9pt}
\end{lemma}
\begin{proof}
Let $K$ be the result of repeatedly expanding $G$ at $u$ along every edge joining $u$ to $H_1$.
$K$~has the form shown at the right of Figure~\ref{figExpandingCycle}.
Also let $K'$ be the result of repeatedly expanding $G'$ ($G$~with edge $e$ subdivided) in an analogous way at $u$.
%It is readily observed that 
Clearly $K'$ is the same as the
graph that results from subdividing the
copies of $e$
in $K$.

Now, for part (a), since $H_1$ is a path with an end at $x_1$, and also $H_2$ is a path with an end at $x_2$, we see that $K'$ is the result of subdividing edges of $K$ that are on endpaths.
Hence, by the first part of Lemma~\ref{lemmaSubdividing1}, we have $\lambda(G')>\lambda(G)$ as required.

For part (b), since $H_1$ is not a path with an end at $x_1$, and nor is $H_2$ a path with an end at $x_2$,
we see that $K'$ is the result of subdividing internal edges of $K$.
Since $K$ is not an $H$ graph, by Lemma~\ref{lemmaSubdividing1}, we have $\lambda(G')<\lambda(G)$ as required.
\end{proof}
If the conditions for parts (a) and (b) of this lemma both fail to be satisfied (i.e. $H_1$ is a suitable path and $H_2$ isn't, or \emph{vice versa}), then
%the arguments used in the proof are unsuccessful 
the proof fails.
This is due to the fact that
expansion leads to at least one copy of $e$ in $K$ being internal
and to another copy of $e$ in $K$ being on an endpath.
Subdivision of the former decreases $\lambda(G)$ whereas subdivision of the latter causes it to increase.
Sometimes, as in Figure~\ref{figGeomSubdivisionEq}, these effects balance exactly; on other occasions one or the other dominates.
We leave a detailed analysis of such cases for later study.

\begin{figure}[t] %[ht]
    $$
    \gctwo{4}{1,1,1}{0,1,1,1}
    \quad
    \raisebox{-0.09in}{\begin{tikzpicture}[scale=0.4]
    \draw [thick] (0,0)--(1,0)--(1.707,0.707)--(2.414,0)--(1.707,-0.707)--(1,0);
    \draw [thick] (2.414,0)--(3.414,0);
    \draw [fill] (0,0) circle [radius=0.15];
    \draw [fill] (1,0) circle [radius=0.15];
    \draw [fill] (2.414,0) circle [radius=0.15];
    \draw [fill] (3.414,0) circle [radius=0.15];
    \draw [fill] (1.707,0.707) circle [radius=0.15];
    \draw [fill] (1.707,-0.707) circle [radius=0.15];
    \end{tikzpicture}}
    \quad\quad\quad\quad
    \gcthree{5}{1,1,0,1}{0,1,1}{0,0,1,1,1}
    \quad
    \raisebox{-0.13in}{\begin{tikzpicture}[scale=0.4,rotate=30]
    \draw [thick] (0,0)--(1,0)--(1.5,0.866)--(2.5,0.866)--(3,0)--(2.5,-0.866)--(1.5,-0.866)--(1,0);
    \draw [thick] (2.5,-0.866)--(3,-1.732);
    \draw [fill] (0,0) circle [radius=0.15];
    \draw [fill] (1,0) circle [radius=0.15];
    \draw [fill] (3,0) circle [radius=0.15];
    \draw [fill] (3,-1.732) circle [radius=0.15];
    \draw [fill] (1.5,0.866) circle [radius=0.15];
    \draw [fill] (2.5,0.866) circle [radius=0.15];
    \draw [fill] (1.5,-0.866) circle [radius=0.15];
    \draw [fill] (2.5,-0.866) circle [radius=0.15];
    \end{tikzpicture}}
    \quad\quad\quad\quad
    \gcfour{6}{1,1,0,0,1}{0,1,1}{0,0,1,1}{0,0,0,1,1,1}
    \quad
    \raisebox{-0.22in}{\begin{tikzpicture}[scale=0.4,rotate=45]
    \draw [thick] (0,0)--(1,0)--(1.383,0.924)--(2.307,1.307)--(3.23,0.924)--(3.613,0)--(3.23,-0.924)--(2.307,-1.307)--(1.383,-0.924)--(1,0);
    \draw [thick] (2.307,-1.307)--(2.307,-2.307);
    \draw [fill] (0,0) circle [radius=0.15];
    \draw [fill] (1,0) circle [radius=0.15];
    \draw [fill] (3.613,0) circle [radius=0.15];
    \draw [fill] (1.383,0.924) circle [radius=0.15];
    \draw [fill] (1.383,-0.924) circle [radius=0.15];
    \draw [fill] (3.23,0.924) circle [radius=0.15];
    \draw [fill] (3.23,-0.924) circle [radius=0.15];
    \draw [fill] (2.307,1.307) circle [radius=0.15];
    \draw [fill] (2.307,-1.307) circle [radius=0.15];
    \draw [fill] (2.307,-2.307) circle [radius=0.15];
    \end{tikzpicture}}
    \vspace{-9pt}
    $$
%\caption{Geometric grid classes of increasing growth rate from left to right, and their row-column graphs}
\caption{Standard figures and row-column graphs of geometric grid classes whose growth rates increase
%with increasing growth rate
from left to right}
\label{figGeomSubdivisionIncr}
%\vspace{-9pt}
    $$
    \gctwo{5}{1,1,1,1}{0,0,1,1,1}
    \quad
    \raisebox{-0.09in}{\begin{tikzpicture}[scale=0.4]
    \draw [thick] (0.293,0.707)--(1,0)--(1.707,0.707)--(2.414,0)--(1.707,-0.707)--(1,0);
    \draw [thick] (0.293,-0.707)--(1,0);
    \draw [thick] (2.414,0)--(3.414,0);
    \draw [fill] (0.293,0.707) circle [radius=0.15];
    \draw [fill] (0.293,-0.707) circle [radius=0.15];
    \draw [fill] (1,0) circle [radius=0.15];
    \draw [fill] (2.414,0) circle [radius=0.15];
    \draw [fill] (3.414,0) circle [radius=0.15];
    \draw [fill] (1.707,0.707) circle [radius=0.15];
    \draw [fill] (1.707,-0.707) circle [radius=0.15];
    \end{tikzpicture}}
    \quad\quad\quad\:\;
    \gcthree{6}{1,1,1,0,1}{0,0,1,1}{0,0,0,1,1,1}
    \quad
    \raisebox{-0.16in}{\begin{tikzpicture}[scale=0.4,rotate=30]
    \draw [thick] (0.234,0.643)--(1,0)--(1.5,0.866)--(2.5,0.866)--(3,0)--(2.5,-0.866)--(1.5,-0.866)--(1,0);
    \draw [thick] (2.5,-0.866)--(3,-1.732);
    \draw [thick] (0.234,-0.643)--(1,0);
    \draw [fill] (0.234,0.643) circle [radius=0.15];
    \draw [fill] (0.234,-0.643) circle [radius=0.15];
    \draw [fill] (1,0) circle [radius=0.15];
    \draw [fill] (3,0) circle [radius=0.15];
    \draw [fill] (1.5,0.866) circle [radius=0.15];
    \draw [fill] (2.5,0.866) circle [radius=0.15];
    \draw [fill] (3,-1.732) circle [radius=0.15];
    \draw [fill] (1.5,-0.866) circle [radius=0.15];
    \draw [fill] (2.5,-0.866) circle [radius=0.15];
    \end{tikzpicture}}
    \quad\quad\quad\:\;
    \gcfour{7}{1,1,1,0,0,1}{0,0,1,1}{0,0,0,1,1}{0,0,0,0,1,1,1}
    \quad
    \raisebox{-0.23in}{\begin{tikzpicture}[scale=0.4,rotate=45]
    \draw [thick] (0.207,0.609)--(1,0)--(1.383,0.924)--(2.307,1.307)--(3.23,0.924)--(3.613,0) --(3.23,-0.924)--(2.307,-1.307)--(1.383,-0.924)--(1,0);
    \draw [thick] (2.307,-1.307)--(2.307,-2.307);
    \draw [thick] (0.207,-0.609)--(1,0);
    \draw [fill] (0.207,0.609) circle [radius=0.15];
    \draw [fill] (0.207,-0.609) circle [radius=0.15];
    \draw [fill] (1,0) circle [radius=0.15];
    \draw [fill] (3.613,0) circle [radius=0.15];
    \draw [fill] (1.383,0.924) circle [radius=0.15];
    \draw [fill] (1.383,-0.924) circle [radius=0.15];
    \draw [fill] (3.23,0.924) circle [radius=0.15];
    \draw [fill] (3.23,-0.924) circle [radius=0.15];
    \draw [fill] (2.307,1.307) circle [radius=0.15];
    \draw [fill] (2.307,-1.307) circle [radius=0.15];
    \draw [fill] (2.307,-2.307) circle [radius=0.15];
    \end{tikzpicture}}
    \vspace{-9pt}
    $$
%\caption{Geometric grid classes of equal growth rate (of~5), and their row-column graphs}
\caption{Standard figures and row-column graphs of geometric grid classes whose growth rates are all the same
(equal to 5)
}
\label{figGeomSubdivisionEq}
%\vspace{-9pt}
\vspace{3pt}
    $$
    \gctwo{6}{1,1,1,1}{0,0,1,1,1,1}
    \quad
    \raisebox{-0.09in}{\begin{tikzpicture}[scale=0.4]
    \draw [thick] (0.293,0.707)--(1,0)--(1.707,0.707)--(2.414,0)--(1.707,-0.707)--(1,0);
    \draw [thick] (0.293,-0.707)--(1,0);
    \draw [thick] (3.121,0.707)--(2.414,0)--(3.121,-0.707);
    \draw [fill] (0.293,0.707) circle [radius=0.15];
    \draw [fill] (0.293,-0.707) circle [radius=0.15];
    \draw [fill] (1,0) circle [radius=0.15];
    \draw [fill] (2.414,0) circle [radius=0.15];
    \draw [fill] (3.121,0.707) circle [radius=0.15];
    \draw [fill] (3.121,-0.707) circle [radius=0.15];
    \draw [fill] (1.707,0.707) circle [radius=0.15];
    \draw [fill] (1.707,-0.707) circle [radius=0.15];
    \end{tikzpicture}}
    \quad\quad\quad
    \gcthree{7}{1,1,1,0,1}{0,0,1,1}{0,0,0,1,1,1,1}
    \quad
    \raisebox{-0.16in}{\begin{tikzpicture}[scale=0.4,rotate=30]
    \draw [thick] (0.234,0.643)--(1,0)--(1.5,0.866)--(2.5,0.866)--(3,0)--(2.5,-0.866)--(1.5,-0.866)--(1,0);
    \draw [thick] (2.326,-1.851)--(2.5,-0.866)--(3.44,-1.208);
    \draw [thick] (0.234,-0.643)--(1,0);
    \draw [fill] (0.234,0.643) circle [radius=0.15];
    \draw [fill] (0.234,-0.643) circle [radius=0.15];
    \draw [fill] (1,0) circle [radius=0.15];
    \draw [fill] (3,0) circle [radius=0.15];
    \draw [fill] (1.5,0.866) circle [radius=0.15];
    \draw [fill] (2.5,0.866) circle [radius=0.15];
    \draw [fill] (2.326,-1.851) circle [radius=0.15];
    \draw [fill] (3.44,-1.208) circle [radius=0.15];
    \draw [fill] (1.5,-0.866) circle [radius=0.15];
    \draw [fill] (2.5,-0.866) circle [radius=0.15];
    \end{tikzpicture}}
    \quad\quad\quad
    \gcfour{8}{1,1,1,0,0,1}{0,0,1,1}{0,0,0,1,1}{0,0,0,0,1,1,1,1}
    \quad
    \raisebox{-0.23in}{\begin{tikzpicture}[scale=0.4,rotate=45]
    \draw [thick] (0.207,0.609)--(1,0)--(1.383,0.924)--(2.307,1.307)--(3.23,0.924)--(3.613,0) --(3.23,-0.924)--(2.307,-1.307)--(1.383,-0.924)--(1,0);
    \draw [thick] (1.698,-2.1)--(2.307,-1.307)--(2.916,-2.1);
    \draw [thick] (0.207,-0.609)--(1,0);
    \draw [fill] (0.207,0.609) circle [radius=0.15];
    \draw [fill] (0.207,-0.609) circle [radius=0.15];
    \draw [fill] (1,0) circle [radius=0.15];
    \draw [fill] (3.613,0) circle [radius=0.15];
    \draw [fill] (1.383,0.924) circle [radius=0.15];
    \draw [fill] (1.383,-0.924) circle [radius=0.15];
    \draw [fill] (3.23,0.924) circle [radius=0.15];
    \draw [fill] (3.23,-0.924) circle [radius=0.15];
    \draw [fill] (2.307,1.307) circle [radius=0.15];
    \draw [fill] (2.307,-1.307) circle [radius=0.15];
    \draw [fill] (1.698,-2.1) circle [radius=0.15];
    \draw [fill] (2.916,-2.1) circle [radius=0.15];
    \end{tikzpicture}}
    \vspace{-9pt}
    $$
%\caption{Geometric grid classes of decreasing growth rate from left to right, and their row-column graphs}
\caption{Standard figures and row-column graphs of geometric grid classes whose growth rates decrease
%with decreasing growth rate
from left to right}
\label{figGeomSubdivisionDecr}
\end{figure}
To conclude, we state the consequent result for the growth rates of geometric grid classes.
To simplify its statement and avoid having to concern ourselves directly with cycle parities, we
define $\grm(M)$ to be $G(M)$ when $G(M)$ has no odd cycles and $\grm(M)$ to be $G(\drm)$ otherwise.

\thmbox{
\begin{cor}
  Suppose $\grm(M)$ is connected.
  \vspace{-3pt}
  \begin{enumerate}[(a)]
    \item If $\grm(M')$ is obtained from $\grm(M)$ by subdividing one or more edges that satisfy the conditions of part (a) of Lemma~\ref{lemmaSubdividing2}, then $\gr(\Geom(M'))>\gr(\Geom(M))$.
    \item If $\grm(M')$ is obtained from $\grm(M)$ by subdividing one or more edges that satisfy the conditions of part (b) of Lemma~\ref{lemmaSubdividing2}, then $\gr(\Geom(M'))<\gr(\Geom(M))$.
  \end{enumerate}
  \vspace{-9pt}
\end{cor}
} %\thmbox

Figure~\ref{figGeomSubdivisionIncr} provides an illustration of part (a) and Figure~\ref{figGeomSubdivisionDecr} an illustration of part (b).

\HIDE{
\textbf{Small growth rates}

only others less than $2+\sqrt5$: cycles, pans: cycle + pendant edge, or $4$-cycle + path with 2 edges

--- Ma~\cite{Ma2001,Ma2005} (in Chinese), see Qiao \& Zhan~\cite{QZ2011}
} %\HIDE

\HIDE{
\textbf{Bounds}
\begin{bullets}
  \item If $\Delta>1$, $\gamma<4\+\Delta-4$ (Heilmann \& Lieb~\cite{HL1972})
  \item $\gamma\geqslant\Delta$ (Lov\'asz \& Pelik\'an~\cite{LP1973})?
  \item $\gamma\geqslant2\overline{d}-1$ (Fisher \& Ryan~\cite{FR1992})
\end{bullets}
Note: $\drm$ has the same maximum and average degree as $M$.
} %\HIDE

%\newpage
% ================================================================
\subsection*{Acknowledgements}
The author would like to thank Penny Bevan
and Robert Brignall
for reading earlier drafts of this paper.
%Her
Their
feedback led to significant improvements in its presentation.
Thanks are also due to an anonymous referee whose comments resulted in a much expanded final section.

\emph{S.D.G.}

% ================================================================
\bibliographystyle{plain}
%{\small\bibliography{mybib}}
{\footnotesize\bibliography{mybib}}

\begin{thebibliography}{10}

\bibitem{AAB2012}
M.~H. Albert, M.~D. Atkinson, and Robert Brignall.
\newblock The enumeration of three pattern classes using monotone grid classes.
\newblock {\em Electron. J. Combin.}, 19(3):~Paper 20, 34 pp. (electronic),
  2012.

\bibitem{AABRV2011}
Michael~H. Albert, M.~D. Atkinson, Mathilde Bouvel, Nik Ru{\v{s}}kuc, and
  Vincent Vatter.
\newblock Geometric grid classes of permutations.
\newblock {\em Trans. Amer. Math. Soc.}, 365(11):5859--5881, 2013.

\bibitem{AAV2014}
Michael~H. Albert, M.~D. Atkinson, and Vincent Vatter.
\newblock Inflations of geometric grid classes: three case studies.
\newblock {\em Australasian J. Combin.}, 58(1):27--47, 2014.

\bibitem{ARV2012}
Michael~H. Albert, Nik Ru\v{s}kuc, and Vincent Vatter.
\newblock Inflations of geometric grid classes of permutations.
\newblock {\em Israel J. Math.}, in press;
  \href{http://arxiv.org/pdf/1202.1833}{arXiv:1202.1833}.

\bibitem{Bevan2013}
David Bevan.
\newblock Growth rates of permutation grid classes, tours on graphs, and the
  spectral radius.
\newblock {\em Trans. Amer. Math. Soc.}, in press;
  \href{http://arxiv.org/pdf/1302.2037}{arXiv:1302.2037}.

\bibitem{CF1969}
P.~Cartier and D.~Foata.
\newblock {\em Probl\`emes combinatoires de commutation et r\'earrangements}.
\newblock Lecture Notes in Mathematics, No. 85. Springer-Verlag, 1969.

\bibitem{CRS2010}
Drago{\v{s}} Cvetkovi{\'c}, Peter Rowlinson, and Slobodan Simi{\'c}.
\newblock {\em An Introduction to the Theory of Graph Spectra}, volume~75 of
  {\em London Mathematical Society Student Texts}.
\newblock Cambridge University Press, 2010.

\bibitem{Elizalde2011}
Sergi Elizalde.
\newblock The {X}-class and almost-increasing permutations.
\newblock {\em Ann. Comb.}, 15(1):51--68, 2011.

\bibitem{Farrell1979}
E.~J. Farrell.
\newblock An introduction to matching polynomials.
\newblock {\em J.~Combin. Theory Ser.~B}, 27(1):75--86, 1979.

\bibitem{FS2009}
Philippe Flajolet and Robert Sedgewick.
\newblock {\em Analytic Combinatorics}.
\newblock Cambridge University Press, 2009.

\bibitem{Godsil1981}
C.~D. Godsil.
\newblock Matchings and walks in graphs.
\newblock {\em J.~Graph Theory}, 5(3):285--297, 1981.

\bibitem{Godsil1993}
C.~D. Godsil.
\newblock {\em Algebraic combinatorics}.
\newblock Chapman \& Hall, 1993.

\bibitem{GG1981}
C.~D. Godsil and I.~Gutman.
\newblock On the theory of the matching polynomial.
\newblock {\em J.~Graph Theory}, 5(2):137--144, 1981.

\bibitem{GS2000}
Massimiliano Goldwurm and Massimo Santini.
\newblock Clique polynomials have a unique root of smallest modulus.
\newblock {\em Inform. Process. Lett.}, 75(3):127--132, 2000.

\bibitem{Gutman1977}
Ivan Gutman.
\newblock The acyclic polynomial of a graph.
\newblock {\em Publ. Inst. Math. (Beograd) (N.S.)}, 22(36):63--69, 1977.

\bibitem{HL1970}
Ole~J. Heilmann and Elliott~H. Lieb.
\newblock Monomers and dimers.
\newblock {\em Phys. Rev. Lett.}, 24:1412--1414, 1970.

\bibitem{HS1975}
Alan~J. Hoffman and John~Howard Smith.
\newblock On the spectral radii of topologically equivalent graphs.
\newblock In {\em Recent Advances in Graph Theory ({P}roc. {S}econd
  {C}zechoslovak {S}ympos., {P}rague, 1974)}, pages 273--281. Academia, 1975.

\bibitem{KK2003}
Tom{\'a}{\v{s}} Kaiser and Martin Klazar.
\newblock On growth rates of closed permutation classes.
\newblock {\em Electron. J. Combin.}, 9(2):~Research paper 10, 20 pp.
  (electronic), 2003.

\bibitem{LP1973}
L.~Lov{\'a}sz and J.~Pelik{\'a}n.
\newblock On the eigenvalues of trees.
\newblock {\em Period. Math. Hungar.}, 3:175--182, 1973.

\bibitem{LP2009}
L{\'a}szl{\'o} Lov{\'a}sz and Michael~D. Plummer.
\newblock {\em Matching Theory}.
\newblock AMS Chelsea Publishing, 2009.
\newblock Corrected reprint of the 1986 original.

\bibitem{MT2004}
Adam Marcus and G{\'a}bor Tardos.
\newblock Excluded permutation matrices and the {S}tanley-{W}ilf conjecture.
\newblock {\em J.~Combin. Theory Ser.~A}, 107(1):153--160, 2004.

\bibitem{Mowshowitz1972}
Abbe Mowshowitz.
\newblock The characteristic polynomial of a graph.
\newblock {\em J.~Combin. Theory Ser.~B}, 12:177--193, 1972.

\bibitem{Riordan2002}
John Riordan.
\newblock {\em An introduction to combinatorial analysis}.
\newblock Dover Publications Inc., 2002.
\newblock Reprint of the 1958 original.

\bibitem{Sachs1964}
Horst Sachs.
\newblock Beziehungen zwischen den in einem {G}raphen enthaltenen {K}reisen und
  seinem charakteristischen {P}olynom.
\newblock {\em Publ. Math. Debrecen}, 11:119--134, 1964.

\bibitem{Vatter2011}
Vincent Vatter.
\newblock Small permutation classes.
\newblock {\em Proc. Lond. Math. Soc.}, 103(5):879--921, 2011.

\bibitem{VW2011}
Vincent Vatter and Steve Waton.
\newblock On partial well-order for monotone grid classes of permutations.
\newblock {\em Order}, 28(2):193--199, 2011.

\bibitem{VW2011b}
Vincent Vatter and Steve Waton.
\newblock On points drawn from a circle.
\newblock {\em Electron. J. Combin.}, 18(1):~Paper 223, 10 pp. (electronic),
  2011.

\bibitem{WatonThesis}
Stephen~D. Waton.
\newblock {\em On permutation classes defined by token passing networks,
  gridding matrices and pictures: Three flavours of involvement}.
\newblock PhD thesis, University of St Andrews, 2007.

\end{thebibliography}

\end{document}